\documentclass[11pt]{amsart}
\usepackage[utf8]{inputenc}
\usepackage{amsmath}
\usepackage{amsthm,amssymb}
\usepackage{cancel} 
\usepackage{tikz}

\makeatletter

\newtheorem{proposition}{Proposition}[section]
\newtheorem{lemma}[proposition]{Lemma}
\newtheorem{theorem}[proposition]{Theorem}
\newtheorem{corollary}[proposition]{Corollary}

\theoremstyle{definition}
\newtheorem{remark}[proposition]{Remark}

\newtheorem{example}[proposition]{Example}
\newtheorem{definition}[proposition]{Definition}

\tikzstyle{place}=[draw,circle,minimum size=1mm,inner sep=1pt,outer sep=-1.1pt,fill=black]

\newcommand{\vertexdef}{\tikzstyle{vertex}=[draw,circle,fill=white,minimum size=1mm,inner sep=0pt]
			\tikzstyle{cv}=[white,draw,circle, 
			minimum size = 1.4mm, inner sep=0pt]}

\usetikzlibrary{shapes.geometric}
\tikzstyle{places}=[draw,rectangle,minimum size=8pt,inner sep=0pt]
\tikzstyle{placesf}=[draw,rectangle,minimum size=5pt,inner sep=0pt]

\tikzstyle{placec}=[draw,circle,minimum size=8pt,inner sep=0pt]
\tikzstyle{placecf}=[draw,circle, minimum size=5pt,inner sep=0pt]

\newcommand{\G}{\mathcal{G}}

\begin{document}
\title{Well-covered Token Graphs}

\author[F.M. Abdelmalek]{F.M. 
Abdelmalek}
\address{Department of Mathematics \& Statistics\\
McMaster University \\,
Hamilton, ON, L8S 4L8, Canada}
\email{abdelf1@mcmaster.ca}

\author[E. Vander Meulen]{Esther Vander Meulen}
\address{Department of Mathematics\\
Redeemer University, Ancaster, ON, L9K 1J4, Canada}
\email{esvandermeulen@redeemer.ca}

\author[K.N. Vander Meulen]{Kevin N. Vander Meulen}
\address{Department of Mathematics\\
Redeemer University, Ancaster, ON, L9K 1J4, Canada}
\email{kvanderm@redeemer.ca}

\author[A. Van Tuyl]{Adam Van Tuyl}
\address{Department of Mathematics \& Statistics\\
McMaster University \\
Hamilton, ON, L8S 4L8, Canada}
\email{vantuyl@math.mcmaster.ca}
\date{\today}

\maketitle

\begin{abstract}
The $k$-token graph  $T_k(G)$ is the graph whose vertices are the $k$-subsets of vertices of a graph $G$, with two
vertices of $T_k(G)$ adjacent if 
their symmetric difference is an edge of $G$. 
We explore when $T_k(G)$ is a well-covered graph,
that is, when all of its maximal independent sets have
the same cardinality.   For bipartite graphs $G$, we classify
when $T_k(G)$ is well-covered.  For an arbitrary 
graph $G$, we show that if $T_2(G)$ is well-covered,
then the girth of $G$ is at most four.  We include
upper and lower bounds on the independence number
of $T_k(G)$, and provide some families of well-covered
token graphs.
\end{abstract}


\section{Introduction}
Let $G$ be a graph with vertex set $V=V(G)$ of order $n$ and let $1\leq k \leq n-1$. 
The \emph{$k$-token graph} of $G$, denoted $T_k(G)$, 
has as vertices the $k$-subsets of $V$ with two vertices adjacent if their symmetric difference
is an edge of $G$. Thus a vertex of $T_k(G)$ can be thought of as a placement of tokens
on $k$ vertices of $G$ with two vertices $u, v \in V(T_k(G))$ adjacent if 
$u$ can be obtained from $v$ by moving a single token along an edge of $G$. Hence,
if $G$ is a connected graph, then $T_k(G)$ is
also connected.  An example of a graph $G$ and its
$2$-token graph $T_2(G)$ is given in Figure \ref{fig:T2}.
We often abuse notation and write $i_1i_2\cdots i_k$ 
instead of $\{i_1,\ldots,i_k\}$ for a $k$-subset of $V$.
Note that $T_1(G)=G$ and that $T_k(G)$ is isomorphic to $T_{n-k}(G)$ for $1\leq k\leq n-1$. 
Thus, when exploring properties of token graphs, 
 it is sufficient to consider values of $k$ satisfying $1\leq k \leq \left\lfloor \frac{n}{2} \right\rfloor$. If
$\{u,v\}$ is an edge of $T_k(G)$, then $\vert u\cap v \vert = k-1$ and we refer to the set $u\cap v$ as the
\emph{anchor} of the edge $\{u,v\}$.
   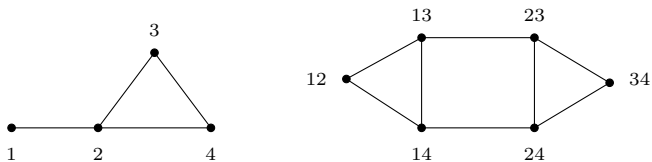
\begin{figure}[ht]
	\begin{center}
	\begin{tiny}
	\begin{tabular}{cc}
		\begin{tikzpicture}
		\vertexdef
        \node[place] (1) at (-0.15,0){}; 
          \node at (-0.15,-0.35){$1$};    
        \node[place] (2) at (1,0){}; 
		  \node at (1,-0.35){$2$};        
        \node[place] (3) at (1.75,1){}; 
          \node at (1.75, 1.3){$3$}; 
        \node[place] (4) at (2.5,0){}; 
          \node at (2.5,-0.35){$4$}; 
        \draw (1) -- (2);
		\draw (2) -- (3);
	    \draw (3) -- (4);
	    \draw (2) -- (4);
	    \end{tikzpicture}  
\qquad & \qquad 
      \begin{tikzpicture}
		\vertexdef
        \node[place] (12) at (0,.65){}; 
          \node at (-0.4,0.65){$12$};        
        \node[place] (14) at (1,0){}; 
		  \node at (1,-0.35){$14$};        
        \node[place] (13) at (1,1.2){}; 
          \node at (1, 1.5){$13$}; 
        \node[place] (24) at (2.5,0){}; 
          \node at (2.5,-0.35){$24$};  
		\node[place] (23) at (2.5,1.2){}; 
          \node at (2.5,1.5){$23$};           
        \node[place] (34) at (3.5,0.6){};   
		  \node at (3.9,0.65){$34$};          
        \draw (12) -- (13);
		\draw (12) -- (14);
	    \draw (13) -- (14);
	    \draw (13) -- (23);
	    \draw (14) -- (24);
	    \draw (23) -- (24);
	    \draw (23) -- (34);
	    \draw (24) -- (34);
	    \end{tikzpicture}
      \end{tabular}	  
      \end{tiny}  
   \end{center}
   \caption{A graph $G$ and its $2$-token graph $T_2(G)$.}\label{fig:T2}
	\end{figure}

The $k$-token graphs appear in the literature under a number of different names.
The $k$-token graphs are a generalization of the \emph{Johnson graphs} (see e.g. \cite{FFHHUW}). In particular,
if $K_n$ is the complete graph on $n$ vertices, then $T_k(K_n)$ is the Johnson graph $J(n,k)$. 
Thus $T_2(K_n)$ is also known to be the complement of the \emph{Kneser graph} $Kn(n,2)$. The $k$-token graph $T_k(G)$
is also known as the \emph{symmetric $k$\textsuperscript{th} power} of $G$ (see e.g. \cite{AGRR, BP}).  Finally, the  $2$-token
graph $T_2(G)$ is also called a \emph{double vertex graph} (see e.g. \cite{ABEL, ALL}). 

Various properties of token
graphs have recently been studied.
For example, in~\cite{CFLR}, 
Carballosa, Fabila-Monroy, Lea\~{n}os, and
Rivera  characterize when the token graphs
are regular, as well as when a token
graph is planar. In~\cite{LT},
Lea\~{n}os and Trujillo-Negrete, prove a conjecture about the connectivity of token graphs. In~\cite{RT}, Rivera and Trujillo-Negrete explore the Hamiltonicity of token graphs. The spectra of token graphs has
been explored in various papers in the 
context of exploring cospectral graphs (see
e.g. \cite{AIP,AGRR}).

In this paper we explore properties of the independent sets of $T_k(G)$, and in particular, we focus on the problem
of determining when $T_k(G)$  is well-covered. An \emph{independent} set
of a graph $\Gamma$ is a subset $S$ of vertices of $\Gamma$ such that no two vertices in $S$ are adjacent in
$\Gamma$. The \emph{independence number} of $\Gamma$,  denoted $\alpha(\Gamma)$,  is the maximum cardinality of any independent set of $\Gamma$. For example, 
for the graphs in Figure~\ref{fig:T2}, $\alpha(G)=\alpha(T_2(G))=2.$   A graph $\Gamma$ 
is \emph{well-covered} if all of its maximal independent sets
have the same cardinality. The graph $G$ in Figure~\ref{fig:T2} is not well-covered
but $T_2(G)$ is well-covered.
The graph $G$ in Figure~\ref{fig:CC} is
not well-covered, nor is $T_2(G)$ well-covered, as illustrated in Example~\ref{ex:CC}. 

Some results are known about $\alpha(T_k(G))$.
In \cite{ACLR}, de Alba, Carballosa,
Lea\~{n}os, and Rivera  bound the independence number of $T_k(G))$ when $G$ is bipartite.
When $k=2$, they derive exact values for
$\alpha(T_2(G))$ when $G$ is either the complete bipartite graph, the cycle $C_n$, or the path $P_n$.   Jim\'enez-Sep\'ulveda and Rivera \cite{JR} determine, $\alpha(T_2(G))$ when $G$ is the fan graph
and the wheel graph.    A sharp lower bound 
on $\alpha(T_2(G))$ appears in work of Deepalakshmi, Marimuthu, Somasundaram, and Arumugam \cite{DMSA} (also see Remark~\ref{rem:lb}).

In the first part of this paper, we derive sharp
upper and lower bounds for $\alpha(T_k(G))$ in
terms of $\alpha(G)$ for all $k\geq 2$.  In particular, in Theorem~\ref{qbound} and Corollary~\ref{cor:aa} we show
that 
$$\binom{\alpha(G)}{k} \leq \alpha(T_k(G)) \leq \frac{1}{k}{\binom{n}{k-1}} \alpha(G).$$
Interestingly, equality in the upper bound depends upon the 
existence of a
specific combinatorial design.  We also
produce a number of methods to construct
 maximal independent sets in $T_k(G)$, when $k=2$, from independent 
sets of $G$ (see e.g. Theorems~\ref{thm:MM}, \ref{thm:VVVV} and \ref{thm:VVE}).
 We obtain some
results about characteristics of
graphs $G$ for which $T_k(G)$ is well-covered.
For example, we provide a classification for bipartite graphs. In particular,
we observe in Corollary~\ref{cor:bpwc} 
that if $G$ be is a bipartite graph, then $T_k(G)$ is well-covered if and only if $k=1$ and 
$G$ is well-covered.
We  determine in 
Corollary~\ref{cor:girth} that a graph $G$ can not have large girth if $T_k(G)$
is well-covered,  where \emph{girth} is the smallest
induced cycle in $G$. We also provide some infinite families of graphs $G$ for which $T_2(G)$ is well-covered.

We use the following outline in our paper. 
In Section~\ref{sec:design} we 
prove our results about the upper bound, while Section~\ref{sec:lower} 
focuses on constructions of maximal independent sets.  This section includes a general lower bound on the independence number.
In Section~\ref{sec:bi}, we characterize when
$T_k(G)$ is well-covered if $G$ is
bipartite.
In Section~\ref{sec:girth} we provide some restrictions on graphs $G$ for which $T_k(G)$ is well-covered. Then
in Section~\ref{sec:wcgraphs} we provide some families of graphs $G$ for which $T_2(G)$ is well-covered.  Section 8 contains some
concluding remarks, and finally, in the 
appendix we list all the graphs $G$
on nine or fewer vertices such that
$T_2(G)$ is well-covered.

We end this section with some common definitions that we will use throughout the paper. 
The subgraph of $G$ \emph{induced} by the set of vertices $A\subset V(G)$ is 
 denoted $G[A]$, having vertex set $A$ with two vertices adjacent in $G[A]$ if and only if they are
adjacent in $G$.
Given a $x \in V(G)$, the \emph{neighbourhood}
of $x$ is the set
$N(x) = \{y~|~ \{x,y\} \in E(G)\}$.  Given
a set $X \subseteq V(G)$, the 
\emph{neighbourhood} of $X$ is $N(X) = 
\bigcup_{x \in X} N(x)$.  The \emph{closed
neighbourhood} of $X$ is $N[X] = X \cup N(X)$.
For any $W \subseteq V(G)$, let $G \backslash W$
denote the graph obtained by removing all the 
vertices of $W$ from $G$ and all edges incident
to a vertex in $W$.

\medskip 
\noindent
{\bf Acknowledgements.}   
Research supported in part by an NSERC USRA (Abdelmalek
and E. Vander Meulen) as well as 
NSERC Discovery Grants 2016-03867 (K.N. Vander Meulen) and
2019-05412 (Van Tuyl).


\medskip


\section{Independent sets of token graphs and 
combinatorial designs}\label{sec:design}

In this section
we describe a relationship between $\alpha(G)$ and $\alpha(T_k(G))$, and a connection 
with combinatorial designs. 
Recall that 
a $t$-$(v,k,\lambda)$ \emph{design} is a collection of $k$-subsets of a set of $v$ elements, such
that every $t$-subset appears in exactly $\lambda$ of the $k$-subsets.

\begin{theorem}\label{qbound} Let $k \leq \left\lfloor{\frac{n}{2}}\right\rfloor$.
If $G$ is a graph on $n$ vertices with no isolated vertices, then
$$\alpha(T_k(G)) \leq \frac{1}{k}{\binom{n}{k-1}} \alpha(G).$$
If equality occurs, then there exists a $t$-$(n,k,\lambda)$ design with $t=k-1$ and $\lambda=\alpha(G)$.
\end{theorem}

\begin{proof}
Consider an independent set $S \subset V(T_k(G))$ with $|S| = \alpha(T_k(G))$. 
Each $v \in S$ contains $k$ potential anchors. 
Consider the multiset $M$ consisting of all the subsets $R$ of cardinality $k-1$ such
that $R\subset v$ for some $v\in S$. 
Then $\vert M \vert = k\alpha(T_k(G))$. 
Note that there are at most ${\binom{n}{k-1}}$ different potential anchors to
be constructed from $n$ vertices and each anchor can appear at most $\alpha(G)$ times
in $M$.   Thus $\vert M \vert \leq {\binom{n}{k-1}}\alpha(G).$ If we have equality, then
every $k-1$ subset must appear as an anchor exactly $\alpha(G)$ times, hence $M$ 
is a $t$-$(n,k,\alpha(G))$ design with $t=k-1$.
\end{proof}
We note that equality is possible in Theorem~\ref{qbound}. For example, let $G=C_5$, the cycle graph
on five vertices. Then $\alpha(G)=2$ and $\{12,23,34,45,15\}$ is a maximum independent set in $T_2(G)$
so that $\alpha(T_2(G))=5$. 

\begin{example}  It was shown in \cite[Cor. 3.10]{ACLR} that $\alpha(T_3(P_{2m+1}))=\frac{(2m+1)m(m+1)}{3}$, which meets the
bound in Theorem~\ref{qbound}. This corresponds to the existence
of a $2$-$(2m+1,3,m+1)$ design. It was also observed in \cite[Cor. 3.10]{ACLR} that $\alpha(T_2(K_{m,m}))=m^2$ which again meets the bound in Theorem~\ref{qbound} and corresponds to the existence of a $1$-$(2m,2,m)$ design.
\end{example}

Another example is the
complete graph $K_n$ with $k=2$ and $n$ even, as noted in the next remark. We provide a proof for completion 
but note that this is a known
result since $T_2(K_n)$ is merely the line graph of the complete graph.
\begin{remark}
If $n\geq 2$, then $\alpha(T_2(K_n))=\left\lfloor \frac{n}{2} \right\rfloor$.
\end{remark}

\begin{proof} If $n$ is even,
the set  $\{12,34,\ldots,n-1n\}$ is 
an independent set of $T_2(K_n)$.
If $n$ is odd, then $\{12,34,\ldots,n-2n-1\}$
is an independent set.  So $\alpha(T_2(K_n))
\geq \left\lfloor \frac{n}{2} \right\rfloor$.
But by Theorem~\ref{qbound}, 
$\left\lfloor \frac{n}{2} \right\rfloor$
is also an upper bound on $\alpha(T_2(K_n))$.
\end{proof}

We can
characterize when we get equality in Theorem~\ref{qbound} for the complete graphs.
Note that 
 $T_k(K_n)$ is isomorphic to the Johnson graph $J(n,k)$. 

\begin{theorem} Given $n\geq 2$ and $k\leq \lfloor \frac{n}{2} \rfloor$, then 
$\alpha(T_k(K_n)) = \frac{1}{k}\binom{n}{k-1}$ if and only if there exists a $t$-$(n,k,1)$ design
with $t=k-1$.
\end{theorem}
\begin{proof}
Let $M$ be a collection of $k$ subsets of an $n$ set forming a $t$-$(n,k,1)$ design with $t=k-1$. 
Then the number of $k$-subsets in $M$ is $\frac{1}{k}\binom{n}{k-1}$ (see e.g.\cite[Cor. 1.4]{HP})
and each element of $M$ is a vertex of $T_k(K_n)$. 
Since each $(k-1)$-subset of $n$ appears in at most one block, no two vertices appearing
in $M$ are adjacent in $T_k(K_n)$. Thus $M$ is an independent set in $T_k(K_n)$. From the proof of Theorem~\ref{qbound}, 
we have $\alpha(T_k(K_n))=\vert M\vert$. The converse follows directly from Theorem~\ref{qbound}.  
\end{proof}

\begin{example} It is known 
that for any $t\geq 1$  there exists a $2$-$(6t+1,3,1)$ design and
a $2$-$(6t+3,3,1)$ design (see Steiner systems, e.g. \cite[p. 174]{HP}). Thus, for $t\geq 1$, 
$$\alpha(T_3(K_{6t+1})) = \frac{1}{3}\binom{6t+1}{2}  \mbox{\rm{\ and\ \  }}  
\alpha(T_3(K_{6t+3})) = \frac{1}{3}\binom{6t+3}{2}.$$
\end{example}

%

\section{Constructing independent sets in token graphs}\label{sec:lower}

In this section, we describe some methods of constructing
independent sets of token graphs.
We start with a remark that describes one way to visualize an
independent set in a $2$-token graph. 
\begin{remark}\label{rem:GKn} For a graph $G$ on $n$ vertices, one can picture an independent set in $T_2(G)$ as 
a set of edges $E$
selected from $K_n$ with the property that no
two adjacent edges in $E$ are part of a triangle whose third side is an edge of $G$ (considering $G$ as a subgraph of $K_n$). 
\end{remark}
To describe some constructions of
independent sets in $k$-token graphs, we introduce the following notation.
Given subsets $V_1,V_2, \ldots,V_k\subseteq V(G)$, 
not necessarily distinct,
we define 
\[V_1V_2\cdots V_k =
\{ x_1x_2\cdots x_k ~|~ x_i \in V_i
~~\mbox{and $x_i \neq x_j$ for all $i \neq j$}\}.\]
Observe that $V_1V_2\cdots V_k$ is a subset of
the vertices of $T_k(G)$.  Indeed, $V(G)V(G)
\cdots V(G)$ ($k$ times) is the set of vertices of $T_k(G)$.

\begin{theorem}\label{thm:MM}
Let $G$ be a graph with independent sets
$V_1,V_2,\ldots,V_k$ such that $V_i \cap V_j = \emptyset$
or $V_i = V_j$ for all $i \neq j$.  
Then $V_1V_2\cdots V_k$ is an
independent set of $T_k(G)$.
\end{theorem}

\begin{proof}
Let $W = V_1V_2\cdots V_k$, and suppose that
$x_1x_2\cdots x_k, y_1y_2\cdots y_k \in W$.
If 
\[| x_1x_2\cdots x_k \triangle y_1y_2\cdots y_k| 
\neq 2,\] then these vertices cannot be
adjacent by the definition of $T_k(G)$. 
Here $A\triangle B$ denotes the symmetric
difference of $A$ and $B$. So, suppose
\[x_1x_2\cdots x_k \triangle y_1y_2\cdots y_k = \{x_i,y_j\}.\] We have $x_i \in V_i$ and $y_j \in V_j$..

If $V_i = V_j$, then $\{x_i,y_j\}$ is not an 
edge in $E(G)$ since $V_i$ is an independent
set, and thus $x_1x_2\cdots x_k$ and $y_1y_2\cdots y_k$
are not adjacent in $T_k(G)$.  So, suppose that 
$V_i \cap V_j = \emptyset$. 
Suppose that $V_i$ appears $a$ times among
$V_1,V_2,\ldots,V_k$, i.e, $V_{i_1} = \cdots =
V_{i_a} = V_i$.  So, exactly $a$ distinct elements
of $\{x_1,x_2,\ldots,x_n\}$ belong to $V_i$ and the same is true for $\{y_1,\ldots,y_n\}$.   However, since 
$x_1x_2 \cdots x_k \triangle y_1y_2\cdots y_k = \{x_i,y_j\}$ with $y_j \in V_j$ 
and $V_i \cap V_j = \emptyset$, all of the 
distinct
elements
in $\{y_1,\ldots,y_k\}$ that belong to $V_i$ 
must appear in $x_1x_2 \cdots x_k
\backslash \{x_i\}$.  But there are only $a-1$  
distinct elements of $V_i$ in $x_1x_2\ldots x_k
\backslash \{x_i\}$.  So, we cannot have 
a symmetric difference of the form $\{x_i,y_j\}$
with $x_i \in V_i$, $y_j \in V_j$ and 
$V_i \cap V_j =\emptyset$.
\end{proof}

\begin{corollary}\label{sizeofalpha}
Let $G$ be a graph with independent sets
$V_1,V_2,\ldots,V_k$ such that $V_i \cap V_j = \emptyset$
or $V_i = V_j$ for all $i \neq j$.  Suppose that
$\{V_{i_1},\ldots,V_{i_l}\}$ are the
distinct subsets that appear among
$V_1,\ldots,V_k$, and that $V_{i_t}$ appears $a_t$ times
(so $a_1+\cdots+a_l=k$).
Then 
\[\binom{|V_{i_1}|}{a_1}\binom{|V_{i_2}|}{a_2}\cdots
\binom{|V_{i_l}|}{a_l} \leq \alpha(T_k(G)).\]
\end{corollary}

\begin{proof}
By Theorem \ref{thm:MM}, it is enough to show
that 
\[\binom{|V_{i_1}|}{a_1}\binom{|V_{i_2}|}{a_2}\cdots
\binom{|V_{i_l}|}{a_l} = |V_1V_2\cdots V_k|.\]
By definition of $V_1\cdots V_k$, $a_t$ of
the elements of $x_1x_2 \cdots x_k \in V_1V_2 \cdots
V_k$ belong to $V_{i_t}$, and these $a_t$ elements
are distinct.  So, there are $\binom{|V_{i_t}|}{a_t}$
ways to pick these $a_t$ elements.  Since $V_{i_t} \cap V_{i_j} = \emptyset$ for all $i\neq j$, the result
now follows.
\end{proof}

We get a lower
bound on the independence number for 
any token graph.

\begin{corollary}\label{cor:aa}
If $G$ is a graph with independence number $\alpha(G)$, then 
\[\binom{\alpha(G)}{k} \leq \alpha(T_k(G)).\]
\end{corollary}

\begin{proof}
Let $W \subseteq V(G)$ be the independent 
set of $G$ with $|W| = \alpha(G)$,
and apply Corollary \ref{sizeofalpha} with
$V_1 = \cdots = V_k = W$.  
\end{proof}

\begin{remark}\label{rem:lb}
The lower bound in Corollary~\ref{cor:aa} is sharp. In particular,   $\alpha(K_{1,n-1})=n-1$ and, in~\cite{ACLR}, it was determined that $\alpha(T_k(K_{1,n-1}))=\binom{n-1}{k}$. When $k=2$ and 
$G$ is not isomorphic to $K_{1,n-1}$, 
then the lower bound in Corollary~\ref{cor:aa} can be improved.  In particular,
$\alpha(T_2(G)) \geq \binom{\alpha(G)}{2} + \left\lfloor \frac{n - \alpha(G)}{2} \right\rfloor$, as first shown in
\cite[Theorem 2.7]{DMSA}.  
\end{remark}

When $k=2$, Theorem \ref{thm:MM} gives us
the following consequences for
the $2$-token graphs of some family of bipartite
graphs.  In particular, we recover
some of the formulas 
in \cite{ACLR}:

\begin{corollary}\label{cor:bipart}
If $G$ is a bipartite graph on $n$ vertices with bipartition $V(G)=V_1 \cup V_2$ such that $|V_1|=|V_2| =\alpha(G)
=\frac{n}{2}$, 
then \[\alpha(T_2(G)) =  \frac{n^2}{4}.\]
In particular, $\alpha(T_2(P_{2n})) = \alpha(T_2(C_{2n}))
= \alpha(T_2(K_{n,n}))= n^2$.
\end{corollary}

\begin{proof}
 By Theorem~\ref{thm:MM}, $V_1V_2$ is an independent
 set of $T_2(G)$, and furthermore, $|V_1V_2| = |V_1||V_2|$
 since $V_1 \cap V_2 = \emptyset$, and thus
$\alpha(T_2(G))\geq |V_1||V_2|$. 
By Theorem~\ref{qbound}, $\alpha(T_2(G)) \leq  \frac {n\alpha(G)}{2} = \frac {2|V_1||V_2|}{2}$.
Therefore $\alpha(T_2(G)) = |V_1||V_2| = \frac{n^2}{4}.$

The last statement follows immediately since each
of the listed bipartite  graphs have $2n$ vertices with
a bipartition $V_1 \cup V_2$ such that $|V_1|=|V_2| =n$.
\end{proof}

The next result follows from Theorem~\ref{thm:MM} by noting that 
 if $V_1,V_2,V_3,V_4,V_5$ are disjoint  sets of $G$, then 
 no vertex of $V_1V_2$ will be adjacent to any vertex in $V_3V_4\cup V_5V_5$
  in $T_2(G)$.  

\begin{corollary}\label{cor:VVV}
Let $G$ be a graph containing disjoint independent
sets $V_1,V_2,\ldots,V_{k}$. If $k$ is even, 
then $V_1V_2 \cup V_3V_4 \cup \cdots V_{k-1}V_{k}$
is an independent set of $T_2(G)$. If $k$ is odd, 
then $V_1V_2 \cup V_3V_4 \cup \cdots V_{k-2}V_{k-1} \cup V_kV_k$ 
is an independent set of $T_2(G)$. 
\end{corollary}

\begin{corollary}
If $G$ is the complete multipartite graph $K_{n_1,n_2,\ldots,n_k}$ of order $n$ with $k$ even,
and if $n_i=\frac{n}{k}$ for $1\leq i \leq k$, then $\alpha(T_2(G))=\frac{n^2}{2k}$. 
\end{corollary}

\begin{proof} By Corollary~\ref{cor:VVV}, 
$\alpha(T_2(G))\geq \frac{n^2}{2k}.$ Equality follows
from Theorem~\ref{qbound}.
\end{proof}

We now give some constructions of
maximal independent sets in $T_2(G)$;  this
enables us to derive lower bounds on $\alpha(T_2(G))$ for specific graphs.
For the following, if $A,B$ are disjoint independents sets
of a graph $G$, we define $$\phi(A,B) = \{\,x \in A \mid B\cup \{x\}  \mbox{\rm \ is an independent set}\}.$$
In Theorem~\ref{thm:VVVV}, we use a partition (in fact, a coloring) of the vertex set of a graph $G$ to obtain a maximal independent set of $T_2(G)$. The 
condition that 
 $\phi(V_j,V_i)=\emptyset$ when $j >i$ implies that the
partition $V_1\cup\cdots\cup V_k$ is constructed so that $V_i$ is a maximal independent set on 
$G\backslash(V_1\cup\cdots\cup V_{i-1})$, 
for $1\leq i\leq k-1$, with $V_0=\emptyset$.

\begin{theorem}\label{thm:VVVV}
Let $G$ be a graph and $V_1 \cup V_2 \cup \cdots \cup V_k$ be a partition
of $V(G)$ into independent sets such that 
 $\phi(V_j,V_i)=\emptyset$ when $j >i$.
If $k$ is even, let 
\begin{eqnarray*}
H &=& \left( V_1V_2 \cup V_3V_4 \cup \cdots \cup V_{k-1}V_k \right) \\
&& \qquad \cup  \left(\phi(V_1,V_2)\phi(V_1,V_2)\cup \cdots \cup \phi(V_{k-1},V_k)\phi(V_{k-1},V_k)\right).\end{eqnarray*}
If $k$ is odd, let 
\begin{eqnarray*}
H &=& \left( V_1V_2 \cup V_3V_4\cup \cdots \cup V_{k-2}V_{k-1}\right) \cup V_kV_k \\
&& \qquad \cup  \left(\phi(V_1,V_2)\phi(V_1,V_2)\cup \cdots \cup \phi(V_{k-2},V_{k-1})\phi(V_{k-2},V_{k-1})\right).
\end{eqnarray*}
Then $H$  is a maximal independent set of $T_2(G).$ 
\end{theorem}

\begin{proof} Using Corollary~\ref{cor:VVV} and the fact that
$V_1V_1\cup V_2V_2\cup \phi(V_1,V_2)$ is an independent set, it follows that $H$ is an independent set. 
We will show that $H$ is maximal.
Suppose $H\cup \{xy\}$ is an independent set for some $xy\in V(T_2(G))$. 
We will demonstrate that $xy\in H$. 

Suppose $x,y\in V_i$ for some $i$. 
We consider three cases. Case 1. Suppose $V_{i-1}V_{i} \subseteq H$. 
Since $\phi(V_{i}V_{i-1}) =\emptyset$, $x$ is adjacent to some
vertex $w\in V_{i-1}$ in $G$. But then $yw$ is adjacent to $yx$ in $T_2(G)$,
contradicting the fact that $H\cup \{xy\}$ is an independent set. Case 2. Suppose 
$V_{i}V_{i+1} \subseteq H$. Then both $x$ and $y$ can not be adjacent to
any vertex in $V_{i+1}$ in $G$, hence $xy\in \phi(V_iV_{i+1})\phi(V_iV_{i+1}) \subseteq H$.  
Case 3. Suppose $V_iV_i \subseteq H$. In this case, $k$ is odd, and $i=k$, in which case
$xy\in H$.  

Suppose $x\in V_i$ and $y \in V_{j}$ for some $j>i$. 
Case 1. Suppose $V_{i-1}V_i \subseteq H$.  Since $H\cup \{xy\}$ is an independent set, 
$xy$ is not adjacent to any vertex in $xV_{i-1}$ in $T_2(G)$. But then $y$ is not
adjacent to any vertex in $V_{i-1}$ in $G$, contradicting the fact that
$\phi(V_j,V_{i-1})=\emptyset$.  
Case 2. Suppose $V_iV_{i+1} \subseteq H$. If $j=i+1$, the $xy \in H$. So suppose $j>i+1$. Then, since $\phi(V_j,V_{i+1})=\emptyset$, $y$ is adjacent to at least one vertex $w\in V_{i+1}$ in $G$. But then $xy$ is adjacent to
$xw \in V_iV_{i+1}$ in $T_2(G)$, contradicting the fact that $H \cup \{xy\}$ is an independent set. 
\end{proof}\vspace{-0.2in}
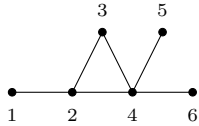
\begin{figure}[ht]
\begin{center}
\begin{tiny}
 \begin{tikzpicture}
	\vertexdef
	 \node[place] (1) at (0,0){}; 
            \node at (0,-0.3){$1$};        
      \node[place] (2) at (0.8,0){}; 
            \node at (0.8,-0.3){$2$};        
      \node[place] (4) at (1.6,0){};  
          \node at (1.6,-0.3){$4$};  
		\node[place] (3) at (1.2,0.8){};  
          \node at (1.2,1.1){$3$};           
        \node[place] (6) at (2.4,0){};    
		  \node at (2.4,-0.3){$6$};          
        \node[place] (5) at (2,0.8){};    
		  \node at (2,1.1){$5$};          
        \draw (1) -- (2);
		\draw (2) -- (3);
	    \draw (2) -- (4);
	    \draw (3) -- (4);
	    \draw (4) -- (6);
	    \draw (4) -- (5);
	    \end{tikzpicture}
      \end{tiny}  
   \end{center}
   \caption{The graph $G$ in Example \ref{ex:CC}.}\label{fig:CC}  
	\end{figure}
\begin{example}\label{ex:CC}
Consider the graph $G$ in Figure~\ref{fig:CC}. Let 
$V_1=\{1,3,5,6\}$, $V_2=\{2\}$ and $V_3=\{4\}.$
Then $\phi(V_1,V_2)=\{5,6\}$ and 
$H=V_1V_2 \cup \phi(V_1,V_2)\phi(V_1,V_2) = 
\{12,23,25,26,56\}$ is a  maximal independent set
in $T_2(G)$ by
Theorem~\ref{thm:VVVV}. 
If $V_1=\{2,5,6 \}, V_2=\{1,4 \}$ and $V_3=\{3\}$,
then 
applying
Theorem~\ref{thm:VVVV}
we get 
$H= \{12,15,16,24,45,46\}$ is an even larger maximal independent set. Further, if $V_1=\{2,5,6\}$, $V_2=\{1,3\}$, and
$V_3=\{4\}$, then $\phi(V_1,V_2)=\{5,6\}$
and $H=\{12,15,16,23,35,36,56\}$ is an even
larger
maximal independent set of $T_2(G)$.
\end{example}
 
The next theorem provides another construction
of a maximal
independent set in $T_2(G)$,
starting with a vertex colouring of $G$. 
\begin{theorem}\label{thm:VVE}
Let $G$ be a graph with a vertex partition into 
independent sets $V_1, V_2, ..., V_k$  such that $\phi(V_j, V_i) = \emptyset$ when $j > i$. 
Let $E$ be a maximal set of edges from $E(G)$ such that:
\begin{enumerate}
    \item\label{c1} If $e=\{u,r\} \in E$ and $u \in V_i$ and $r \in V_j$, then, $e$ is an isolated edge in 
    $G[V_i \cup V_j]$.
    \item\label{c2} If $e_1, e_2 \in E$ share a common endpoint in $G$, then there is no triangle in $G$ containing $e_1$ and $e_2$. 
\end{enumerate}
Then $A = V_1V_1 \cup V_2V_2 \cup \cdots\cup V_kV_k \cup E$ is a maximal independent set in $T_2(G)$.
\end{theorem}

\begin{proof}
We first show that $A$ is an independent set in $T_2(G)$. 
The subset $A \backslash E$ is an independent set since each $V_i$ is an independent set. 
Let $x = v_{ia}v_{ib} \in A \backslash E$ with $v_{ia},v_{ib} \in V_i$. 
Let $y \in E$. If $x,y$ do not share an anchor, then $x$ and $y$ are not adjacent. 
Thus, without loss of generality suppose $y = v_{ia}u$. 
By definition of $E$, $u \in V_j$ for some $j \neq i$. Further $\{v_{ia},u\}$ is an isolated edge in $G[V_i \cup V_j]$. 
Thus, $y$ is not adjacent to $x$ in $T_2(G)$.
 
Suppose now that $x,y \in E$. If $x,y$ do not share an anchor, then $x$ and $y$ are not adjacent. So suppose 
$x = uv$ and $y = ut$. By condition~(\ref{c2}),  $v$ and $t$ are not adjacent in $G$, 
and hence $x$ and $y$ are not adjacent in $T_2(G)$. Therefore $A$ is an independent set.

Now we show $A$ is maximal. 
Suppose $x = x_ix_j \not \in A$ for some $x_i \in V_i$ and $x_j \in V_j$. 
We know that $i \neq j$ (since otherwise, $x$ would be in $A$). Without loss of generality $i < j$. 

Suppose that no vertex of $A$ is adjacent to $x$ in $T_2(G)$. 
Then $x_i$ is the only possible neighbour of $x_j$ in $V_i$, and $x_j$ is the only possible neighbour of $x_i$ in $V_j$. Indeed, 
suppose $x_i$ has a neighbour $z \neq x_j$ with $z \in V_j$. Then $x_ix_j$ is adjacent to $x_jz \in V_jV_j \subseteq A$. Similarly, suppose 
$x_j$ has a neighbour $y \neq x_i$ and $y \in V_i$. Then, $x_ix_j$ is adjacent to $x_iy \in V_iV_i \subseteq A$.
Since $\phi(V_j, V_i) = \emptyset$, it follows that $x\in E$ and
thus $x$ is an isolated edge in $G[V_i\cup V_j]$. But this contradicts condition~(\ref{c1})
 given that
$E$ is maximal. Thus $x$ has a neighbour in $A$.  Therefore $A$ is a maximal independent set.
\end{proof}
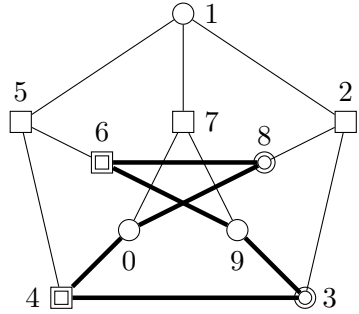
\begin{figure}[ht]
\centering
\begin{tikzpicture}[scale = 1.8]
\node (1) at (0,1.8)[placec][label=right:1]{};
\node (2) at (1.2,1)[places][label=above:2]{};
\node (3) at (0.9,-0.3)[placecf][label=right:3]{};
\node (3b) at (0.9,-0.3)[placec]{};
\node (4) at (-0.9,-0.3)[places][label=left:4]{};
\node (4b) at (-0.9,-0.3)[placesf]{};
\node (5) at (-1.2,1)[places][label=above:5]{};
\node (6) at (-0.6,0.7)[places][label=above:6]{};
\node (6b) at (-0.6,0.7)[placesf]{};
\node (7) at (0,1)[places][label=right:7]{};
\node (8) at (0.6,0.7)[placecf][label=above:8]{};
\node (8b) at (0.6,0.7)[placec]{};

\node (9) at (0.4,0.2)[placec][label=below:9]{};
\node (0) at (-0.4,0.2)[placec][label=below:0]{};
\draw [] (1) to (2);
\draw [] (1) to (5);
\draw [] (1) to (7);
\draw [] (2) to (3);
\draw [] (2) to (8);
\draw [line width=0.6mm] (3) to (4);
\draw [line width=0.6mm] (3) to (9);
\draw [] (4) to (5);
\draw [line width=0.6mm] (4) to (0);
\draw [] (5) to (6);
\draw [line width=0.6mm] (6) to (8);
\draw [line width=0.6mm] (6) to (9);
\draw [] (7) to (0);
\draw [] (7) to (9);
\draw[line width=0.6mm] (8) to (0);
\end{tikzpicture}
\caption{A colouring of the Petersen graph with some bolded edges.}\label{fig:Pete}
\end{figure}

\begin{example} Let $G$ be the Petersen graph 
with independent sets $V_1 = \{0,1,9\}$, $V_2 = \{2,5,7\}$, $V_3 = \{4,6\}, V_4 = \{3,8\}$
and bolded edges $E=\{04,43,39,96,68,08\}$,
depicted in Figure~\ref{fig:Pete}. Note
that $\phi(V_j,V_i)=\emptyset$ for $j>i$.
    By Theorem~\ref{thm:VVE}, $V_1V_1 \cup V_2V_2 \cup V_3V_3 \cup V_4V_4 \cup E$ 
    is a maximal independent set of $T_2(G)$ of cardinality fourteen. 
     
     Note that $\phi(V_1,V_2) = \emptyset$ and $\phi(V_3, V_4) = \emptyset$. Thus, by
    Theorem~\ref{thm:VVVV}, $V_1V_2 \cup V_3V_4$ 
    is a maximal independent set of $T_2(G)$ of cardinality thirteen. 
Thus,  the $2$-token graph of the Petersen graph is not well-covered.

Since $G$ contains no triangles, the edges of $G$ form an independent set in $T_2(G)$ 
of cardinality fifteen. This set is maximal since the addition of any further
edge would form a triangle with the edges of $G$
(see Remark~\ref{rem:GKn}).

Further, if we take $U_1=\{5,8,9\}$, $U_2=\{0,2,6\}$, $U_3=\{3\}$, $U_4=\{4\}$, $U_5=\{1\}$, and $U_6=\{7\}$
and $F=\{39,23,04,45,15,12,07,79\}$, then one can check that the hybrid construction $U_1U_1\cup U_2U_2 \cup U_3U_4 \cup U_5U_6 \cup F$
is an independent set of $T_2(G)$ with cardinality sixteen. A computer check can verify that
$\alpha(T_2(G))=16.$   
\end{example}


\section{Characterization of bipartite graphs $G$ with $T_k(G)$ well-covered}\label{sec:bi}

In this section we characterize when
$T_k(G)$ is well-covered if $G$ is a connected
bipartite graph. Before we address bipartite graphs, we 
present the following result  based upon
\cite[Proposition 1]{Z}, which we will find useful.

\begin{theorem}\label{wellcovered}
Suppose $G$ is a graph and $I$ is an independent
set of $V(G)$ such that $I$ is not maximal.
If $G \backslash N[I]$ is not well-covered, then
$G$ is not well-covered.
\end{theorem}

\begin{proof}
Since $I$ is not maximal, $G \backslash N[I]$
is not the empty graph.
We now prove the contrapositive statement.
Suppose that $W_1,W_2$ are two maximal 
independent sets in $G \backslash N[I]$.
Then $W_1 \cup I$ and $W_2 \cup I$ are independent
sets in $G$, and furthermore, they must
be maximally independent.  But since $G$ is well-covered, $|W_1 \cup I| = |W_2 \cup I|$, which
implies that $|W_1|=|W_2|$, i.e., $G\backslash N[I]$ is
well-covered.
\end{proof}

A bipartite graph with bipartition 
$V(G) = L \cup R$ is \emph{balanced} if $|L|=|R|$.

\begin{theorem}\label{dealbabipartition}\cite{ACLR}
and \cite[Proposition 12]{FFHHUW}
Let $G$ be a connected bipartite graph with bipartition $V(G)=L\cup R$. 
Then $T_k(G)$ is bipartite with bipartition 
$$    \{A \subseteq V(T_k(G))\ \mid \  |R \cap A| \, \text{is even}\ \} 
\cup    
    \{A \subseteq V(T_k(G))\ \mid \  |R \cap A| \, \text{is odd}\ \}.$$
\end{theorem}

We require a sequence of technical lemmas.

\begin{lemma}\label{swaplemma}
Let $G$ be a bipartite graph with $A_1,A_2 \in V(T_k(G))$. If $|A_1\cap L|=\ell$  
and $A_2$ is adjacent to $A_1$ in $T_k(G)$, 
then $|A_2\cap L|\in \{\ell-1,\ell+1\}.$
\end{lemma}

\begin{proof}
Since $A_1,A_2$ are adjacent, they must share an anchor. 
Hence $|A_2\cap L| = \ell+i$ for some $i \in \{-1,0,1\}$. However, if $i = 0$, that would contradict Theorem~\ref{dealbabipartition}.
\end{proof}

\begin{lemma}\label{typek0degree}
Suppose $G$ is a connected bipartite graph with bipartition $V(G) = L \cup R$ and $|L| \geq k\geq 2$. 
If $A \in V(T_k(G))$ with 
$|A\cap L| = k$, then ${\rm deg}(A) \geq 2$.
\end{lemma}
\begin{proof}
Suppose $|A\cap L|=k$ and $A = A' \cup \{x,y\}\subseteq L$ for some $A'$ with $|A'|=k-2$. Since $G$ is
connected, let $u,v\in R$ be adjacent to $x$ and $y$, respectively.  {Note that $u$ and $v$
need not be distinct.}
Then, $A$ is adjacent to both $A' \cup \{x,v\}$ and $A' \cup \{u, y\}$. 
\end{proof}

\begin{lemma}\label{typek0notisolated}
Let $G$ be a connected bipartite graph with bipartition 
$V(G) = L \cup R$ and let $k \geq 2$. 
Suppose that $|L| \geq k$ and $|R| \geq 1$. 
If $A_1, A_2, \in V(T_k(G))$,  $A_1\neq A_2$ and
$|A_1\cap L| = |A_2\cap L| = k$,  then 
$A_1$ is not an isolated vertex 
in $T_k(G)\backslash N[A_2]$. 
\end{lemma}
\begin{proof}
By Lemma~\ref{swaplemma}, $A_1$ is not adjacent to $A_2$.
Since $A_1 \neq A_2$, without loss of generality suppose $A_1 = l_1l_2\cdots l_{k}$ and 
$l_1\not\in A_2$. Since $G$ is connected, there is a vertex $r\in R$ such that
$l_2$ is adjacent to $r$ in $G$. Thus 
$A_1$ is adjacent to $l_1rl_3\cdots l_{k}$. However, 
$A_2$ is not adjacent to $l_1rl_3\cdots l_{k}$ since $l_1,r\not\in A_2$. 
Thus $l_1rl_3\cdots l_{k} \not\in N[A_2]$. Therefore
$A_1$ is not an isolated vertex in $T_k(G)\backslash N[A_2]$.
\end{proof}

\begin{lemma}\label{lem:nbrs}
Let $G$ be a connected bipartite graph with bipartition 
$V(G) = L \cup R$ and let $k \geq 2$. Suppose that $|L| \geq k$ and $|R| \geq 2$. 
Let $A= l_1l_2l_3\cdots l_k \in V(T_k(G))$ with $|A\cap L|=k$
and $B \in V(T_k(G))$ with $|B\cap L| = k-2$. 
If $B$ is an isolated vertex in $T_k(G) \backslash N[A]$, then $|B\cap A|=k-2$ and,
after relabelling, $B = r_1r_2l_3\cdots l_k$ such that:
\begin{enumerate}
    \item\label{NL} $N(\{l_3,l_4,\ldots,l_k\}) \subseteq \{r_1,r_2\}$.
    \item\label{NR} $N(\{r_1,r_2\}) \subseteq \{l_1,l_2,l_3,\ldots,l_k\}$.
    \item\label{deg} ${\rm deg}(B) \geq 2$ in $T_k(G)$.
\end{enumerate}
\end{lemma}
\begin{proof}
Recall that since $G$ is connected, $T_k(G)$ is connected. 
In particular, $T_k(G)$ contains no isolated vertices. 
Suppose $B$ is an isolated vertex in $T_k(G) \backslash N[A]$.
Then $N(B) \subseteq N(A)$ and $A,B$ must share at least one common neighbour. 
By Lemma~\ref{swaplemma}, there is a vertex $C\in V(T_k(G))$ with $|C\cap L|=k-1$ adjacent 
to $B$ and $A$. 
Since $C$ must share an anchor with $A$, without loss of generality $C = r_1l_2l_3\cdots l_k$ for some $r_1\in R$. 
Since $B$ shares an anchor with $C$, by Lemmas~\ref{swaplemma} and~\ref{typek0degree}, $|B\cap A|=k-2$. Without loss of generality assume
that $B = r_1r_2l_3\cdots l_k$ for some $r_2\in R$.

Suppose $N(\{l_3,l_4,\ldots,l_k\}) \not\subseteq \{r_1,r_2\}$. Without loss of generality, 
suppose $l_3$ is adjacent to some $x \not \in \{r_1,r_2\}$. Then $B$ has a 
neighbour $r_1r_2xl_4\cdots l_k$, but this vertex is not adjacent to $A$ by Lemma~\ref{swaplemma}. But then $B$ is not isolated in $T_k(G) \backslash N[A]$, a contradiction.
Therefore $N(\{l_3,l_4,\ldots,l_k\}) \subseteq \{r_1,r_2\}$.

Suppose $N(\{r_1,r_2\}) \not\subseteq \{l_1,l_2,\ldots,l_k\}$. Without loss of generality,
 suppose $r_2$ is adjacent to some $x \not \in \{l_1,l_2,\ldots,l_k\}$. Then 
 $B$ is adjacent to $r_1xl_3...l_k$, which is not adjacent to $A$, 
contradicting the fact that $B$ is isolated in $T_k(G) \backslash N[A]$.
Therefore $N(\{r_1,r_2\}) \subseteq \{l_1,l_2,\ldots,l_k\}$.

Finally, $B$ is adjacent to both $C$ and $l_1r_2l_3\cdots l_k$. 
Therefore deg$(B)\geq 2$.
\end{proof}

\begin{lemma}\label{deleteB} 
Let $G$ be a connected bipartite graph with bipartition $V(G) = L \cup R$ and let $k \geq 2$. 
Suppose that $|L| \geq k$, $|R| \geq 2$ 
{and $|L|+|R| \geq 5$}.
Suppose $A= l_1l_2l_3\cdots l_k \in V(T_k(G))$  with $|A\cap L|=k$ and 
$B \in V(T_k(G))$ with $|B\cap L|=k-2$. 
If $B$ is an isolated vertex in $T_k(G) \backslash N[A]$,
then there is 
no isolated vertex in $T_k(G)\backslash N[B]$. 
\end{lemma}
\begin{proof}
Suppose 
$B$ is an isolated vertex in $T_k(G) \backslash N[A]$.
First, note that by Lemma~\ref{lem:nbrs}, we can assume $B = r_1r_2l_3\cdots l_k$
and $N(B)\subseteq N(A)$. 
Suppose that 
$C$ is an isolated vertex in $T_k(G)\backslash N[B]$,
{and so $N(C) \subseteq N(B)$}.

By Lemma~\ref{swaplemma}, $|C\cap L|\in\{k,k-2\}$. We claim that $|C \cap L| = k-2$.
Suppose that $|C \cap L| = k$.  We first note
that $C \neq A$.  Indeed, if $C =A$, then
$N(C) = N(B) = N(A)$.  If $k > 2$,
let $r$ be any vertex adjacent to $l_k$. Then, $A = l_1l_2\cdots l_k$ is adjacent to $W = l_1l_2\cdots l_{k-1}r$. Note that $W$ cannot be adjacent to $B$, since $l_1,l_2$ both appear in $W$, while neither appear in $B$. So $W$ and $B$ do not share an anchor, and thus
$N(B) \neq N(A)$.  If $k =2$,
then $A = l_1l_2$, and $B = r_1r_2$. By Lemma~\ref{lem:nbrs} we have $N(\{r_1,r_2\}) \subseteq \{l_1,l_2\}$. If $C = A$, then $N(C) \subseteq N(A)$ and consequently $N(\{l_1,l_2\}) \subseteq \{r_1,r_2\}$. But this means $\{l_1,l_2,r_1,r_2\}$ is a maximal connected component of $G$, contradicting the hypothesis that $G$ is connected on $\geq 5$ vertices.  Therefore $C \neq A$.

Since $N(C) \subseteq N(B) \subseteq N(A)$,  $C$ is  isolated in $T_k(G) \setminus N[A]$, contradicting Lemma~\ref{typek0notisolated}. Thus $|C \cap L| \neq k$,  and by Lemma~\ref{swaplemma}, $|C \cap L| = k-2$.

Since 
$C$ must is isolated in $T_k(G) \backslash N[A]$, 
 $C$ and $B$ have a common neighbour $D$ which is also adjacent to $A$. Thus, given that
$A$ and $D$ have a common anchor, without loss of generality, 
$D=r_1xl_3\ldots l_k$ for some $x\in\{l_1,l_2\}$.

Suppose $x=l_1$. Then since $D$ and $C$ have a common anchor, without loss of
generality, either $C=r_1yl_3\cdots l_k$ for some $y\in R$ or 
$C=r_1l_1zl_4\cdots l_k$ for some $z\in R$. Note that if $k=2$, the latter
case does not occur. 
Suppose $C= r_1yl_3\cdots l_k$. 
Then $y\neq r_2$ since $B\neq C$ and further, $C$ is adjacent to
$E=l_1yl_3\ldots l_k$. However, $E$ is not adjacent to $B$ since $|E\cap B|=k-2$, 
violating the fact that $N(C)\subseteq N(B)$. Therefore.   
$C=r_1l_1zl_4\ldots l_k$. Then $z=r_2$ by Lemma~\ref{lem:nbrs}.\ref{NL}. 
 In this case, $C$ is adjacent to
$F=r_1l_2l_1l_4\cdots l_k$ and hence cannot be adjacent to $B$ since
$|F\cap B|=k-2$.
This again violates the fact that $N(C)\subseteq N(B)$. 
Therefore $x=l_2$.
However, with $x=l_2$, a similar argument also leads to a contradiction.
Therefore there is no vertex $C$ that is isolated in $T_k(G) \backslash N[B]$.
\end{proof}

We pause to give an example that will be used
to simplify our proof of Theorem~\ref{bipartitethem}.

\begin{example}\label{stargraph}
Let $a\geq 4$ and $G = K_{1,a}$ be the complete bipartite
graph (a \emph{star} graph) with bipartition
$V =\{x\}\cup \{y_1,\ldots,y_a\}$.  We show that
$T_k(G)$ is not well-covered for any
$1 \leq k \leq \lfloor \frac{a+1}{2} \rfloor$.
Since $T_1(G)=G$ is not well-covered, we first 
consider $2 \leq k < \lfloor \frac{a+1}{2} \rfloor$.
Note that 
the vertices of $V(T_k(G))$
come in two types: a $k$ subset of $\{y_1,\ldots,
y_a\}$, and
a $k$ subset of $V$ that contains $x$ and
a $k-1$ subset of $\{y_1,\ldots,y_a\}$.  In fact,
these two sets form the bipartition of $T_k(G)$.
There are $\binom{a}{k}$ vertices of the first type,
and $\binom{a}{k-1}$ vertices of the second type. 
When $k\neq \frac{a+1}{2}$,
the parts of the bipartition have different cardinalities,
and hence  $T_k(G)$ is not well-covered. 
So, suppose $k = \frac{a+1}{2}$, and hence
$\binom{a}{k} = \binom{a}{k-1}$.  
If we take the non-maximal independent set
$I = \{y_1y_2\cdots y_k\}$, then
the bipartite graph
$T_k(G) \backslash N[I]$ is not well-covered
since \[N[I] = \{y_1y_2\cdots y_k,
xy_2\cdots y_k, xy_1y_3\cdots y_k,
\ldots, xy_1y_2\cdots y_{k-1}\},\]
and so one part has $\binom{a}{k}-1$
elements, and the other has $\binom{a}{k-1}-k$.
By Theorem \ref{wellcovered}, $T_k(G)$ 
is not well-covered.
\end{example}

\begin{theorem}\label{bipartitethem}
Let $2\leq k \leq \lceil\frac{n}{2}\rceil$.
If $G$ is a connected bipartite graph 
with $|G| = n \geq 5$,
then $T_k(G)$ is not well-covered. 
\end{theorem}
\begin{proof}
Suppose $G$ is a connected bipartite graph on $n$ vertices with bipartition $V(G) = L \cup R$.
Without loss of generality, assume that $|L|\geq |R|$ and 
hence $|L|\geq k$, since $2\leq k \leq \lceil\frac{n}{2}\rceil$. 

Suppose $|R| = 1$. Then $G$ is the star graph which
is not well-covered by
Example~\ref{stargraph}.

 Suppose $|R| \geq 2$. 
 If $T_k(G)$ is well-covered, it is necessary that the bipartition
 as described in Theorem~\ref{dealbabipartition} is balanced.
    Let $A \in V(T_k(G))$ with $|A\cap L|=k$. Let $Q =T_k(G) \backslash N[A]$.
    Note that  $\rm{deg}(A) \geq 2$ by Lemma~\ref{typek0degree}.
    If $Q$ contains no isolated vertices, then $Q$ is a bipartite graph with nonempty bipartitions
    that is not balanced and hence $Q$ is not well-covered. Therefore, by
    Theorem~\ref{wellcovered}, $T_k(G)$ is not well-covered.
    Suppose $Q$ contains an isolated vertex $B$.
By Lemma~\ref{swaplemma},  $|B\cap L|=k-2$. Then, by Lemmas~\ref{typek0notisolated} and~\ref{deleteB}, 
it follows that $W=T_k(G)\backslash N[B]$ contains no isolated vertices, 
and  $\deg(B) \geq 2$. Thus $W$ is an unbalanced bipartite graph with nonempty bipartitions
 and no isolated vertices. Therefore, by
    Theorem~\ref{wellcovered}, $T_k(G)$ is not well-covered.
\end{proof}

The following corollary gives the desired 
characterization.

\begin{corollary}\label{cor:bpwc}
Let $G$ be a connected bipartite graph with $|G| = n$. Given $1\leq k\leq \lceil\frac{n}{2}\rceil$,
then $T_k(G)$ is well-covered if and only if 
$G$ is well-covered and $k=1$ $($in this case $T_k(G) \cong G)$.
\end{corollary}

\begin{proof}
If $n \geq 5$, then the result follows from Theorem
\ref{bipartitethem}.  A direct computation on
all bipartite graphs on four or fewer vertices 
finishes the proof.
\end{proof}


\section{Restrictions on graphs   with well-covered
2-token graphs.}\label{sec:girth}

In this section, we derive  some restrictions on $G$, with
regard to girth and independence number, when
$T_k(G)$ is well-covered. 

\begin{theorem}\label{thm:4-3cyc}
 Suppose $|G| \geq 3$, $G$ is connected and $T_2(G)$ is well-covered. If $P = \{x_1,x_2,x_3\}$ is an induced path 
 in $G$, then either $P$ is part of a four cycle 
 or at least one of the vertices of $P$ is part of a $3$-cycle in $G$.
\end{theorem}
\begin{proof}
Suppose that no vertex of $P$ is part of a $3$-cycle in $G$ and that there is no induced four cycle 
in $G$ that includes the
vertices $x_1,x_2$, and $x_3$.

Let $I_1,I_2,I_3$ be the vertices of $H = G \backslash P$ that are adjacent to $x_1,x_2,x_3$ respectively. Since  $x_1,x_2,x_3$ are not part of a triangle in $G$, we know that $I_1,I_2,I_3$ are independent sets in $G$. Likewise,  $I_1 \cap I_2 = \emptyset = I_2 \cap I_3$ since  $x_1,x_2,x_3$ are not part of a triangle. Further,
$I_1 \cap I_3 = \emptyset$ since none
of $x_1,x_2$, and $x_3$ are part of a $4$-cycle.

Consider $T_2(P) = \{x_1x_2, x_1x_3, x_2x_3\}$. In $T_2(G)$,  $N(T_2(P))$  is precisely the set of vertices $x_iI_j = \{x_iy ~|~ y \in I_j\}$
for $i,j \in \{1,2,3\}$ with $i \neq j$. 

Suppose $I_1\neq \emptyset$ and $I_3\neq \emptyset$. 
Consider the independent set $A = x_1I_1 \cup x_2I_2 \cup x_3I_3 \cup I_1I_3$. We have $N(T_2(P)) \subseteq N[A]$, and $T_2(P) \cap N[A] = \emptyset$. Thus, $T_2(P)$ is a maximal connected component of $T_2(G) \backslash N[A]$. Since $T_2(P)$ is not well-covered,  it follows that $T_2(G)$ is not well-covered by Theorem~\ref{wellcovered}.
 
 Suppose $I_1= \emptyset$ and$I_3\neq \emptyset$. Consider the independent set $A = x_2I_2 \cup x_3I_3$. In $T_2(G) \backslash N[A]$, we have $N(T_2(P)) = x_1I_3$.  Let $B$ be a maximal independent set in $T_2(G) \backslash (N[A] \cup T_2(P) \cup x_1I_3)$. Then $A\cup B$ is an independent set and no vertex of $B$ is adjacent to $T_2(P)$. 
    Thus 
    $T_2(P) \cup x_1(I_3 \backslash N[B])$ as a maximal connected component
    in $T_2(G) \backslash N[A\cup B]$. Further, $T_2(P) \cup x_1(I_3 \backslash N[B])$ is not well-covered. To see this, consider that  $\{x_1x_3\}$ is a maximal independent set, 
    but $\{x_1x_2, x_2x_3\}$ is a larger independent set. Therefore, by Theorem~\ref{wellcovered}, $T_2(G)$ is not well-covered.
    
    Suppose
    $I_1\neq \emptyset$ and $I_3= \emptyset$. By symmetry, this case is similar the previous case.
    
    Suppose 
    $I_1=I_3=\emptyset$. In $T_2(G)$, we have $N(T_2(P)) = x_1I_2 \cup x_3I_2$. Consider the independent set $A = x_2I_2$. Note that $N(T_2(P)) \subseteq N[A]$ while $T_2(P) \cap N[A]=\emptyset$. Thus, $T_2(G) \backslash N[A]$ contains an isolated path $T_2(P)$ and hence is not well-covered.
    Therefore, $T_2(G)$ is not well-covered by Theorem~\ref{wellcovered}.
\end{proof}

\begin{corollary}\label{cor:girth}
If $|G| \geq 3$, $G$ is connected and $T_2(G)$ is well-covered, then ${\rm girth}(G) \leq 4$.
\end{corollary}

We finish this section with a bound on $\alpha(G)$ when $T_2(G)$ is well-covered.

\begin{lemma}\label{lemma:phi}
Let $G$ be a connected graph, and let $V_1,V_2,\ldots,V_r$ form a partition of $V(G)$ such that $V_i$ is a maximum independent set in $G \backslash \left(\cup_{j \leq i} V_j\right)$. If $T_2(G)$ is well-covered, then $|\phi(V_{2i-1},V_{2i})| \leq 1$
for $1\leq i\leq \left\lfloor \frac{r}{2}\right\rfloor.$
\end{lemma}
\begin{proof}
 Suppose that   
$|\phi(V_1,V_2)| \geq 2$. Since $V_1$ is a maximum independent set of $G$, 
and because $V_2 \cup \phi(V_1,V_2)$
is an independent set,
$|V_2 \cup \phi(V_1,V_2)| \leq |V_1|$. Thus $|V_2| \leq |V_1| - 2$.

Let $x \in \phi(V_1,V_2)$.  Consider the sets $W_1 = V_1 \backslash \{x\}$ and $W_2 = V_2 \cup \{x\}$. Note that $|W_1 W_2| = (|V_1|-1)(|V_2|+1) \geq |V_1V_2|$ since $V_2 \leq V_1 - 2$. Further $W_1 W_2 \cup \phi(V_1,V_2)\phi(V_1,V_2)$ is an independent set in $T_2(G)$. 
By Theorem~\ref{thm:VVVV}, $A = V_1V_2 \cup \phi(V_1,V_2)\phi(V_1,V_2) \cup V_2V_3\cdots$ is a maximal independent set of $T_2(G)$. But 
$\vert A\backslash V_1V_2 \cup W_1 W_2 \vert 
> \vert A\vert,$  and hence 
$T_2(G)$ is not well-covered. 
By similar argument,  
$|\phi(V_{2i-1},V_{2i})| \leq 1$ for $2 \leq i \leq \lfloor \frac{r}{2} \rfloor$. 
\end{proof}

\begin{theorem}\label{cor:parts}
Suppose $G$ is connected and $T_2(G)$ is well-covered. Let $V_1,V_2,...,V_r$ 
form a partition of $V(G)$ such that $V_i$ is a maximum independent set in $G \backslash \left(\cup_{j \leq i} V_j\right)$.
Then for $1\leq k\leq \left\lfloor\frac{r}{2}\right\rfloor,$
$$|V_{2k-1}| \leq \left\lfloor\frac{1+2|V_{2k}| + \sqrt{8|V_{2k}|+1}}{2} \right\rfloor\leq \left\lfloor1+|V_{2k}|+\sqrt{2|V_{2k}|}\right\rfloor.$$
\end{theorem}

\begin{proof}
We begin by noting that Lemma \ref{lemma:phi}
implies that $\phi(V_{2i-1},V_{2i})\phi(V_{2i-1},V_{2i}) = \emptyset$ for $i=1,\ldots,\lfloor \frac{r}{2} 
\rfloor$.  Consequently, by Theorem 
\ref{thm:VVVV}, the set 
$V_1V_2 \cup V_3V_4 \cup \cdots \cup V_{r-1}V_r$
if $r$ even, or $V_1V_2 \cup \cdots V_{r-2}V_{r-1}
\cup V_rV_r$ if $r$ odd, is a  maximal independent 
set of $T_2(G)$.

Let $\mathcal{B}_k$
be the bipartite  subgraph of $G$ induced
by $V_{2k-1}\cup V_{2k}.$  For simplicity, consider the case $k=1$. We first claim
that $\alpha(T_2(\mathcal{B}_1))=\vert V_{1}V_{2} \vert.$
Suppose for contradiction that $A$ is a maximal independent set of $\mathcal{B}_1$ with $\vert A\vert >|V_{1}V_{2}|$. Then, $T_2(G)$ contains an independent set with
cardinality $|A \cup V_3V_4 \cup \cdots| > |V_1V_2 \cup V_3V_4 \cup \cdots|$, which contradicts the hypotheses that $T_2(G)$ is well-covered or that $V_1V_2 \cup V_3V_4 \cup \cdots $ is maximal. 
Therefore 
$\alpha(T_2(\mathcal{B}_1))=\vert V_{1}V_{2} \vert$.

Because $V_1V_1 \cup V_2V_2$ is an independent
set of $\mathcal{B}_1$, it follows that 
$\alpha(T_2(\mathcal{B}_1)) = 
\vert V_1V_2 \vert = |V_1||V_2| \geq \binom{|V_1|}{2} + \binom{|V_2|}{2}$. 
Solving this inequality for $|V_1|$ (and using the fact that $|V_1|$ must be an integer) gives $|V_1| \leq \left\lfloor\frac{1+2|V_2| + \sqrt{8|V_2|+1}}{2} \right\rfloor \leq \left\lfloor1+|V_2|+\sqrt{2|V_2|}\right\rfloor$.
The cases with $k\neq 1$ are similar.
\end{proof}

\begin{corollary}\label{cor:alphabound}
Suppose $G$ is connected and $|G| = n \geq 3$. If $T_2(G)$ is well-covered, then $\alpha(G) \leq \left.\lfloor \frac{n-1+\sqrt{n-1}}{2} \rfloor\right.$.
\end{corollary}
\begin{proof}
Partition the vertices of $G$ into sets $V_1, V_2, \ldots, V_r$ such that $|V_1| = \alpha(G)$ and $|V_2| = \alpha(G \setminus V_1)$. Since $T_2(G)$ is well-covered and $G$ is connected, Corollary~\ref{cor:bpwc} implies that $G$ is not bipartite. Thus, there is at least one vertex which is not in $V_1 \cup V_2$ and so $|V_2| \leq n-\alpha(G)-1$. Applying $|V_1| = \alpha(G)$ and $|V_2| = n-\alpha(G)-1$ to the  inequality from the Theorem ~\ref{cor:parts} yields the required result. 
\end{proof}

\section{Constructions of well-covered token graphs}\label{sec:wcgraphs}

In this section we describe some graphs $G$ for which $T_2(G)$ is well-covered. Many of the graphs fit within a certain family of graphs that we describe in Definition~\ref{rem:genbb}. 

We first note that there is no direct inheritance with respect to being well-covered for token graphs. 
If $G$ is well-covered, then
there is no guarantee that $T_k(G)$ is well-covered.
For example, the cycle $C_4$ is well-covered but   $T_2(C_4)$
is isomorphic to the complete bipartite graph $K_{2,4}$ which
is not well-covered. 
There are also graphs for which $G$ is not well-covered but $T_{2}(G)$ is well-covered, as observed in Figure~\ref{fig:T2} (and, for example, Theorem~\ref{thm:bba} with $s=m$ and $t=0$).

\begin{theorem}\label{thm:complete} 
For $n\geq 2$, $T_{2}(K_{n})$ is well-covered.
\end{theorem}
\begin{proof}
Let $A$ be an independent set of vertices in $T_2(K_n)$. No vertex of $K_n$ appears in
more than one pair in $A$. 
 If there exists $i,j\in V(K_{n})$, with neither $i$ nor $j$ appearing in any pair in $A$, then $A\cup \{ij\}$ is also an independent set. It follows that if $A$ is a maximal independent set, 
 then $|A|= \left\lfloor\frac{n}{2}\right\rfloor$. Thus $T_{2}(K_{n})$ is well-covered.
\end{proof}

While $T_2(K_n)$ is well-covered, we expect that in general $T_k(K_n)$ is not
well-covered for $k>2$. For example, it is known that 
for each $n\geq 9$, there exists a partial Steiner triple system of order $n$  that does not
have an embedding of order $v$ for any $v< 2n + 1$, demonstrating the existence of a maximal
independent set in $T_3(K_n)$ that is not maximum when $n\equiv 1, 3 \mod 6$ (see \cite{H} 
and \cite{CR}).
 For example,  the maximal independent set $\{123,367,345,147,256\}$ in $T_3(K_7)$ cannot be completed to become a Fano plane.

In this section, we use the fact that if $H$ is a subgraph of $G$, then $T_2(H)$ is a subgraph
of $T_2(G)$.  Taking a maximal independent set of a graph $G$ and considering its restriction
to a subset of vertices $A$ of $G$ will give an independent set in $G[A]$.  
\begin{remark}\label{rem:parts}
If $V(G)=V_1\cup V_2$ with $V_1\cap V_2 = \emptyset$ then $\alpha(G)\leq \alpha(G[V_1])+\alpha(G[V_2])$ 
\end{remark}

Using a computer search, we found that there were few graphs $G$ for which $T_2(G)$ is well-covered for 
small $n$ (see the Appendix). Besides the complete graphs, many of the graphs $G$ we found for which $T_2(G)$ is well-covered were in a class $\G$ described below.

\begin{definition}\label{rem:genbb}
Define $\G$ to be the set of graphs obtained by taking the disjoint union of $K_m$ and $K_n$, $n \geq m$, and inserting some edges.
An example of a graph in $\G$ is given in Figure~\ref{fig:zig}. Let $X=V(K_{m})=\{x_1,x_2,...,x_m\}$ and $Y=V(K_{n})=\{y_1,y_2,...,y_n\}$. Let $G\in\G$ and $H=T_2(G)$. Then the vertices of $H$
can be partitioned as $V(H)=XX\cup XY\cup YY$ with
$H[XX]=T_2(K_m)$,
$H[YY]=T_2(K_n)$, and
$H[XY]=K_m\square K_n$ (the
Cartesian product of $K_m$ and
$K_n$).  
Further, if $x_i$ is adjacent to $y_k$ in $G$, then $H$ contains the edges $\{\{x_j x_i,x_j y_k\}| 1\leq j \leq m, j\neq i\}$ 
and $\{\{x_i y_\ell,y_k y_\ell\}| 1\leq \ell \leq n, \ell\neq k\}$.
\end{definition}

In next
set of theorems we determine classes of token graphs $T_2(G)$
with $G \in \G$ that are well-covered and classes that are not well-covered. 
We start by considering the
independence number for some of the graphs in $\G$.

\begin{lemma}\label{lem:bbalpha}
Let $G\in\G$ with 
at most $n-m$ vertices of $Y$ having a neighbour in $X$. Then $\alpha(T_2(G))=\alpha(T_2(K_m))+\alpha(T_2(K_n))+m$.
\end{lemma}
\begin{proof}
By Remark~\ref{rem:parts},
$\alpha(T_2(G))\leq     
\alpha (T_2(K_m))+\alpha (K_m\square K_n)+\alpha (T_2(K_n))$. If $n=m$, then
$T_2(G)$ is just the disjoint union of $T_2(K_m)$, $K_m\square K_n$ and $T_2(K_n)$ and so equality holds in
the previous inequality. Suppose that $n>m$.  Note that $\alpha (K_m\square K_n)=m$.  
Without loss of generality, assume that none of the vertices $y_1,y_2,\ldots,y_m$ are adjacent
to any of the vertices of $K_m$. 
Let $A=\{ x_1x_2,x_3x_4,\ldots\}$ be a maximal independent set of $T_2(K_m)$,
$B=\{ y_1y_2,y_3y_4,\ldots \}$ be a maximal independent set of $T_2(K_n)$,
and $C=\{x_1y_1,x_2y_2,\ldots,x_{m-1}y_{m-1} \}$. Let $D=\{x_my_m\}$ if $m$ is even
and $D=\{x_my_{m+1}\}$ if $m$ is odd. Then $A\cup B\cup C \cup D$ is an independent
set of $T_2(G)$. 
Therefore $\alpha(T_2(G))=\alpha(T_2(K_m))+\alpha(T_2(K_n))+m$.
\end{proof}

\begin{remark}
The independent set constructed in the proof of Lemma~\ref{lem:bbalpha} can be constructed via the construction of 
Corollary~\ref{cor:VVV}.
In particular, taking $V_{2i-1}=\{x_{2i-1},y_{2i}\}$
and $V_{2i}=\{x_{2i},y_{2i-1}\}$ for $1\leq i \leq \left\lfloor\frac{m}{2}\right\rfloor$,
with $W=\{x_m,y_m\}$ if $m$ is odd and $W=\emptyset$ if $m$ is even,
and taking 
$U_{i}=\{y_i\}$ for $m+1\leq i\leq n$, then 
$A\cup B\cup C\cup D$ is
the same as 
$\{V_1V_2\cup V_3V_4\cup \cdots \}\cup \{U_{m+1}U_{m+2}
\cup U_{m+3}U_{m+4} \cup \cdots \} \cup WW$. While the tools of Section~\ref{sec:lower} are helpful for  constructing independent sets in token graphs, in this section, such as in the previous proof, we give more direct descriptions of some independent sets.
\end{remark}

For graphs $G\in \G$,  
the next three theorems provide forbidden configurations for $T_2(G)$  
 to be well-covered.
The restriction of at most $n-m$ vertices of $K_m$ having a neighbour in $K_n$ allows us to provide the exact value of the independence number in Lemma~\ref{lem:bbalpha}. As such, we focus on graphs in $\G$ having this  
restriction as we develop the next results.  The next theorem provides a parity restriction for such graphs in $\G$ that 
are well-covered.

\begin{theorem}\label{thm:evenleaf}
Let $G\in\G$ be such that at most
$n-m$ vertices of $Y$ have a neighbour in $X$. If $G$ is connected and either $n$ or $m$ is even, then $T_2(G)$ is not well-covered.
\end{theorem}
\begin{proof} Since $G$ is connected, we
know that $n>m.$ 
By Lemma~\ref{lem:bbalpha},
$\alpha(T_2(G))= \alpha(T_2(K_m))
+\alpha(T_2(K_n))+m.$
Suppose $m$ is even.  
Suppose there is a vertex, say $y_n$, which is
adjacent to every vertex in $K_m$. Let $I$
be any maximal independent set of $T_2(G)$ with
$x_1y_n\in I.$ Then $I$ can contain no
edge $x_1x\in XX$. Thus $|I\cap XX|<\alpha(T_2(K_m))$ and hence
$G$ is not well-covered.

Suppose there is a vertex in $Y$  
adjacent to some vertex in $X$ but not adjacent to every vertex in $X$. Assume that $y_n$ is adjacent to $x_1$ but not $x_m$. Let $I$ be a maximal independent set
with $\{y_nx_m\}\cup \{x_2x_3,\ldots,x_{m-2}x_{m-1}\}\subseteq I.$
Note that $x_1x_m\not\in I$ 
since $y_nx_m\in I$ and $x_1$ is adjacent to
$y_n$ in $G$. Thus $|I\cap XX|<\alpha(T_2(K_m))$.
Therefore $T_2(G)$ is not well-covered.
The case with $n$ even is similar to the previous case.
\end{proof}

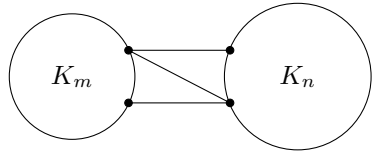
\begin{figure}[ht]
	\begin{center}
	\begin{small}	
		\begin{tikzpicture}
		\vertexdef
		\node (circle) [draw, circle, minimum size = 16.7mm] at (0, 0) {};
		\node at (0,0) {$K_m$};        
		\node (circle) [draw, circle, minimum size = 19.5mm] at (3, 0) {};
         \node at (3,0) {$K_n$};       
         \node[place] (1) at (0.75,0.35){};        
         \node[place] (3) at (0.75,-0.35){};        
         \node[place] (2) at (2.1,0.35){};
         \node[place] (4) at (2.1,-0.35){};
        \draw (1) -- (2);
		\draw (1) -- (4);
	    \draw (3) -- (4);
	    \end{tikzpicture}  
      \end{small}  
   \end{center}
   \caption{A graph $G\in\G$ with $T_2(G)$ not well-covered (Theorem~\ref{thm:zigzag}).}\label{fig:zig}
	\end{figure}

\begin{theorem}\label{thm:zigzag}
Let $G\in \G$ 
with 
$n\geq m+2$,
with $\{x_1, y_1\}$, $\{x_1,y_2\}$, 
$\{x_2, y_2\} \in E(G)$ and at most $n-m$ vertices of $Y$ have a neighbour in $X$.
Then $T_2(G)$ is not well-covered.  
\end{theorem}

\begin{proof}
By Lemma~\ref{lem:bbalpha},  $\alpha(T_2(G))=\alpha(T_2(K_m))+\alpha(T_2(K_n))+m$.
And if $I$ is an independent set with $|I|=\alpha(T_2(G))$, then $I$ must contain
 $\alpha (T_2(K_m))$ vertices from $T_2(K_m)$.
If $m\geq 3$, consider a maximal independent set $I$ of $T_2(G)$ containing the vertices $x_2y_1$ and $x_3y_2$, 
as well as the vertices $x_mx_{m-1},x_{m-2}x_{m-3},\ldots,x_{t+1}x_{t}$ for $t\in \{3,4\}$ (depending
on the parity of $m$).
Then 
$I$ cannot include the vertices $x_2x_3$, $x_1x_3$, and $x_1x_2$ since these vertices are all adjacent to either $x_2y_1$ or $x_3y_2$.   
If $m=2$, construct a maximal independent set containing $x_2y_1$, and hence $x_1x_2\not\in I.$ 
In either case, $\vert I\cap XX\vert < \alpha(T_2(K_m))$, and hence
$|I| < \alpha(T_2(G))$. 
Thus $T_2(G)$ is not well-covered.
\end{proof}

\begin{figure}[ht]
	\begin{center}
	\begin{small}	
		\begin{tikzpicture}
		\vertexdef
		\node (circle) [draw, circle, minimum size = 16.7mm] at (0, 0) {};
		\node at (0,0) {$K_m$};        
		\node (circle) [draw, circle, minimum size = 19.5mm] at (3, 0) {};
         \node at (3,0) {$K_n$};       
         \node[place] (1) at (0.75,0.35){};
         \node[place] (3) at (0.75,-0.35){};        
         \node[place] (5) at (0.78,0){};
         \node[place] (6) at (2.05,0){};
                  
         \node[place] (2) at (2.1,0.35){};
         \node[place] (4) at (2.1,-0.35){};
  
        \draw (1) -- (2);
		\draw (5) -- (6);
	    \draw (3) -- (4);
	    \end{tikzpicture}  
      \end{small}  
   \end{center}
   \caption{A graph $G\in \G$ with $T_2(G)$ not well-covered (Theorem~\ref{thm:3stripes}).}
	\end{figure}
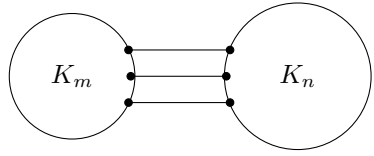

\begin{theorem}\label{thm:3stripes}
Let $G\in \G$ 
with 
$n\geq m+3$,
 with  $\{x_1,y_1\}$, $\{x_2,y_2\}$, $\{x_3,y_3\}\in E(G)$ and at most $n-m$ vertices of $Y$ have a neighbour in $X$. 
Then $T_2(G)$ is not well-covered.
\end{theorem}
\begin{proof}
By Lemma~\ref{lem:bbalpha},   $\alpha(T_2(G))=\alpha(T_2(K_m))+\alpha(T_2(K_n))+m$.
Consider a maximal independent set, say $I$, of $T_2(G)$ containing the vertices $x_1y_3$, $x_2y_1$, and $x_3y_2$. Suppose also 
$x_{2i-2}x_{2i-1}\subseteq I$ for $3\leq i\leq \left\lceil \frac{m}{2} \right\rceil$. 
Note that $x_1y_3$ is adjacent to $x_1x_3$, $x_2y_1$ is adjacent to $x_1x_2$, and 
$x_3y_1$ is adjacent to $x_2x_3$ in $T_2(G)$. 
Thus $x_1x_3$, $x_1x_2$, and $x_2x_3$ are not in $I$.
Therefore $\vert I\cap XX \vert < \alpha(T_2(K_m))$ and hence 
$\vert I \vert < \alpha(T_2(G))$. Thus $T_2(G)$ is not well-covered.
\end{proof}

In the context of the previous two theorems, if $G\in\G$ is well-covered, then the edges between $Y$ and $X$ in $G$ must consist
of at most two distinct stars, and if there are two stars, they must be disjoint.

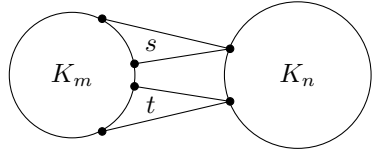
\begin{figure}[ht]
	\begin{center}
	\begin{small}	
		\begin{tikzpicture}
		\vertexdef
		\node (circle) [draw, circle, minimum size = 16.7mm] at (0, 0) {};
		\node at (0,0) {$K_m$};        
		\node (circle) [draw, circle, minimum size = 19.5mm] at (3, 0) {};
         \node at (3,0) {$K_n$};       
         \node[place] (1) at (0.83,0.15){};     
	      \node[place] (a) at (0.4, 0.75){};

		\node at (1.05, 0.4) {$s$};	
         \node[place] (3) at (0.83,-0.15){};        
         \node[place] (5) at (0.4,-0.75){};
                      
		\node at (1.05, -0.4) {$t$};    
         \node[place] (2) at (2.1,0.35){};
         \node[place] (4) at (2.1,-0.35){};
  
        \draw (a) -- (2);
        \draw (1) -- (2);
		\draw (5) -- (4);
	    \draw (3) -- (4);
	    \end{tikzpicture}  
      \end{small}  
   \end{center}
   \caption{A graph $G\in \G$ with $T_2(G)$ well-covered if $s+t\leq m$ (Theorem~\ref{thm:bba}).}
	\end{figure}

In the following theorem we consider graphs in $G\in\G$ having
one vertex of $X$ adjacent to $s$ vertices of $Y$ and
another adjacent to $t$ other vertices of $Y$, to get a
well-covered graph $T_2(G)$ when $s+t\leq m$. 

\begin{theorem}\label{thm:bba}
Let $G\in \G$ 
with $n>m$, both odd,  
such that $N[y_1]=\{x_1,x_2,\ldots,x_s\}\cup Y$ 
$N[y_2]=\{x_{s+1},x_{s+2},\ldots,x_{s+t}\} \cup Y$
and  $N[y_i]=Y$ for all $i$, $3\leq i\leq n$, with $0\leq s+t \leq m$. Then $T_2(G)$ is well-covered.
\end{theorem}
\begin{proof} 
By Lemma~\ref{lem:bbalpha}, 
$\alpha(T_2(G))= \alpha(T_2(K_m))+\alpha (T_2(K_n))+m$.  Let $I=A\cup B\cup C$ be a maximal independent set in $T_2(G)$ with $A\subseteq YY$, $B\subseteq XY$, $C\subseteq XX$. It is enough to show that $|A|=\alpha (T_2(K_n))$, $|B|=m$, and $|C|=\alpha (T_2(K_m))$.

Suppose $|A|<\alpha (T_2(K_n))$. Then there are at least three vertices
$y_a,y_b,y_c\in Y$ which do not appear in any pair of $A$.  
Without loss of generality $y_c\not\in\{y_1,y_2\}$. Suppose $y_c$ appears in a pair $xy_c$ of $B$. Note that
$x$ has at most one neighbour in $Y$. 
Thus, without loss of generality, $x$ is not adjacent to $y_b$. In this case, and the case
when $y_c$ appears in no pair of $B$, $I=\{y_by_c\}\cup I$ is an independent set. 
But then $I$ is not maximal. Therefore, $|A|=\alpha(T_2(K_n))$. 

Suppose $|B|< m$. Then there is some $x\in X$ that appears in no pair of $B$.  
Also, there are at least $n-m+1\geq 3$ vertices of $Y$ that are not part of any pair in $B$; say $y_a,y_b,y_c$. 
Let $Z=\{ y_a,y_b,y_c \}$.
We claim that $H=\{ {x}y\} \cup I$ is an independent set of $T_2(G)$ for some $y\in Z$. 
If ${x}$ does not appear in any pair in $C$, then there will be one less restriction 
on the possible $y\in Z$ (to ensure $H$ is an independent set),
so assume ${x}w \in C$ for some $w\in X$. Then $w$ could be adjacent to $y_1$ or $y_2$ but not both. Thus
there is at most one ${x}y$ adjacent to ${x}w$ in $T_2(G)$ for $y\in Z$. Without loss of generality,
assume $w$ is adjacent to $y_a$. Now, if either $y_b$ or $y_c$ does not appear in any pair of $A$, then 
 $H$ is an independent set for that $y$. 
 
 Suppose that $y_by_c \in A$. Now ${x}$ 
is adjacent to at most one of $y_b$ and $y_c$. 
If $x$ is adjacent to $y_b$, then let $y=y_b$, otherwise let $y=y_c$. In either case, $H$ is an independent set. 

Suppose that $y_by_c \not \in A$ but
$y_by_r, y_cy_q \in A$.
Again ${x}$ is adjacent to at most one
of $y_r$ and $y_q$. Without loss of generality, assume that $y_q$ is not adjacent to ${x}$.
Then take $y=y_c$, and $H$ is an independent set.  Since in each case, 
$H$ is constructed to be an independent set, this would imply that $I$ is not maximal. Therefore
we conclude that $|B|=m$. 

Suppose $|C|<\alpha (T_2(K_m))$. Then there are at least 3 vertices $x_a,x_b,x_c\in X$
that do not appear in any pair of $C$. 
If $x_ay_1\not\in B$, $x_ay_2\not\in B$, $x_by_1\not\in B$ and $x_by_2\not\in B$, 
then $\{x_{a}x_b\}\cup I$ is an independent set in $T_2(G)$. 
Without loss of generality, assume $x_ay_1\in B$. Note that then $x_ay_2 \not \in B$ since
$y_1$ is adjacent to $y_2$. Likewise, $x_by_1 \not\in B$.

Suppose $x_by_2\not\in B$. 
Then $\{x_{b}x_c\}\cup I$ is an independent set in $T_2(G)$.
In either case, $I$ is not maximal. 

Suppose $x_by_2\in B$. Since $x_c$ is not adjacent to
both $y_1$ and $y_2$, assume that $x_c$ 
is not adjacent to $y_1$. Then $\{x_ax_c\} \cup I$ is an independent
set in $T_2(G)$, and so $I$ is not maximal.
Thus $|C|=\alpha (T_2(K_m))$. 

Therefore $T_2(G)$ is well-covered. 
\end{proof}

\begin{figure}[ht]
	\begin{center}
	\begin{small}	
		\begin{tikzpicture}
		\vertexdef
		\node (circle) [draw, circle, minimum size = 16.7mm] at (0, 0) {};
		\node at (0,0) {$K_m$};        
		\node (circle) [draw, circle, minimum size = 19.5mm] at (3, 0) {};
         \node at (3,0) {$K_n$};       
         \node[place] (1) at (2.05,0.15){};        
	      \node[place] (a) at (2.38, 0.75){};

		\node at (1.75, 0.4) {$s$};	
		
         \node[place] (3) at (2.05,-0.15){};        
         \node[place] (5) at (2.38,-0.75){};
                      
		\node at (1.75, -0.4) {$t$};      
         \node[place] (2) at (.77, 0.35){}; 
         \node[place] (4) at (0.77, -0.35){}; 
        \draw (a) -- (2);
        \draw (1) -- (2);
		\draw (5) -- (4);
	    \draw (3) -- (4);
	    \end{tikzpicture}  
      \end{small}  
   \end{center}
   \caption{A graph $G\in \G$ with $T_2(G)$ well-covered if $s+t\leq n-m$ (Theorem~\ref{thm:bbb}).}
	\end{figure}
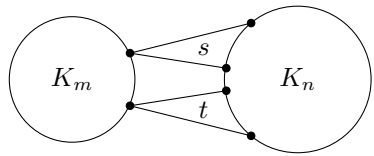

We next consider the graph considered in Theorem~\ref{thm:bba} with the stars between $K_m$ and $K_n$ reversed.

\begin{theorem}\label{thm:bbb}
Let $G\in \G$ 
with $n>m$, both odd,  
such that $N[x_1]=\{y_1,y_2,\ldots,y_s\}\cup X$, 
$N[x_2]=\{y_{s+1},y_{s+2},\ldots,y_{s+t}\} \cup X$ 
and  $N[x_i]=X$ for all $i$, $3\leq i\leq m$, with $0\leq s+t \leq n-m$. Then $T_2(G)$ is well-covered.
\end{theorem}
\begin{proof}  
As in the proof of Theorem~\ref{thm:bba}, we let $I=A\cup B\cup C$ be a maximal independent set in $T_2(G)$ with $A\subseteq YY$, $B\subseteq XY$, $C\subseteq XX$. It is enough to show that $|A|=\alpha (T_2(K_n))$, $|B|=m$, and $|C|=\alpha (T_2(K_m))$.
By similar arguments to those used for the proof of Theorem~\ref{thm:bba}, we can show that $|A|=\alpha (T_2(K_n))$ and $|C|=\alpha (T_2(K_m))$.

Suppose $|B|< m$. Without loss of generality, there is some vertex 
$x \in X$ that belongs to no pair in $B$. Also, there are at least $n-m+1\geq s+t+1$ vertices
of $Y$ that belong to no pair in $B$. 

Suppose $x\in\{x_1,x_2\}$. 
Without loss of generality,
$x=x_1$. If $xy\in XY$ is adjacent to a vertex
of $C$, then $y\in\{y_{s+1},\ldots,y_{s+t}\}$. Thus
as most $t$ vertices of $XY$ containing $x$ 
are adjacent to a vertex in $C$. 
At most $s$ vertices of $A$ contain a member in
$\{y_1,\ldots,y_{s}\}$. Thus 
at most $s$ vertices of $XY$ containing $x$ are adjacent to vertices in $A$. 
Since there are at least $s+t+1$ vertices of $Y$ that belong to no pair in $B$, there is some $v\in Y$ such that $xv$ is not adjacent to any vertex of $I$.
Hence $\{xv\}\cup I$ is an independent set in $T_2(G)$. %

Suppose $x\not\in\{x_1,x_2\}$. Since there are at least $s+t+1$ vertices
of $Y$ that do not appear in any pair of $B$, there is some $y_j$
with $j>s+t$ such that $y_j$ does not appear in any pair of $B$. 
Hence $\{xy_j\}\cup I$ is an independent set in $T_2(G)$. 

Since $I$ is maximal, it follows that $|B|=m$.
Thus $T_2(G)$ is well-covered.
\end{proof}

\begin{figure}[ht]
	\begin{center}
	\begin{small}	
		\begin{tikzpicture}
		\vertexdef
		\node (circle) [draw, circle, minimum size = 16.7mm] at (0, 0) {};
		\node at (0,0) {$K_m$};        
		\node (circle) [draw, circle, minimum size = 19.5mm] at (3, 0) {};
         \node at (3,0) {$K_n$};   	
		 \node[place] (2) at (2.1,0.35){};
		  \node[place] (1) at (0.83,0.15){};     
	      \node[place] (a) at (0.4, 0.75){};

		\node at (1.05, 0.4) {$s$};	

         \node[place] (3) at (2.05,-0.15){};       
         \node[place] (5) at (2.38,-0.75){};     
                      
		\node at (1.75, -0.4) {$t$};     
         \node[place] (4) at (0.77, -0.35){}; 
        \draw (a) -- (2);
        \draw (1) -- (2);
		\draw (5) -- (4);
	    \draw (3) -- (4);
	    \end{tikzpicture}  
      \end{small}  
   \end{center}
   \caption{A graph $G\in \G$ with $T_2(G)$ well-covered if $t+1\leq n-m$ and $s+1\leq m$ (Theorem~\ref{thm:bbc}).}
	\end{figure}
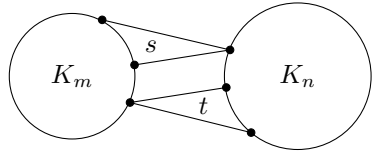

In the next theorem we consider 
graphs $G\in\G$ with one vertex of $X$ adjacent to 
$t$ vertices of $Y$ and one vertex of $Y$ 
 adjacent to $s$ vertices of $X$. Due to Theorem~\ref{thm:zigzag}, these stars will need to be disjoint if $T_2(G)$ is well-covered.
 
\begin{theorem}\label{thm:bbc}
Let $G\in\G$ with $n>m$, both odd, such that
$N[x_1]=\{y_2,y_3,\ldots,y_{t+1}\}\cup X$,  
$N[y_1]=\{x_2,x_3,\ldots,x_{s+1}\}\cup Y$,
$N[x_i]=X$ for $s+2\leq i\leq m$ and
$N[y_i]=Y$ for $t+2\leq i\leq n$,
with $s+1\leq m$ and $t+1\leq n-m$.
Then $T_2(G)$ is well-covered.
\end{theorem}
\begin{proof} 
By Lemma~\ref{lem:bbalpha}, 
$\alpha(T_2(G))= \alpha(T_2(K_m))+\alpha (T_2(K_n))+m$.  Let $I=A\cup B\cup C$ be a maximal independent set in $T_2(G)$ with $A\subseteq YY$, $B\subseteq XY$, $C\subseteq XX$. It is enough to show that $|A|=\alpha (T_2(K_n))$, $|B|=m$, and $|C|=\alpha (T_2(K_m))$.

Suppose $|A|<\alpha (T_2(K_n))$. Then there are at least three vertices  $y_a,y_b,y_c \in Y$ that are not in any pair of $A$.
At least two vertices in $\{y_ax_1,y_bx_1,
 y_cx_1\}$ cannot be in  $B$; without loss of generality, $y_ax_1,y_bx_1\not \in B$. 
 
 Suppose $y_1x_1\in B$. Then $y_a\neq y_1$ and $y_b\neq y_1$. Hence $\{y_ay_b\}\cup I$ is an independent set.
 
 Suppose $y_1x_1\not \in B$. If $y_a\neq y_1$ then $\{y_ay_b\}\cup I$ is an independent set. Suppose $y_a=y_1$. If $y_bw\not \in B$ for all $w\in N(y_1)$, then $\{y_ay_b\}\cup I$ is an independent set. 
 Likewise,
 if $y_cw\not \in B$ for all $w\in N(y_1)$, then $\{y_ay_c\}\cup I$ is an independent set.
 Finally, if $y_bw\in B$ for some $w\in N(y_1)$ and $y_ct\in B$ for some $t\in N(y_1)$, 
 then $\{y_by_c\}\cup I$ is an independent set. Therefore, if $|A|<\alpha(T_2(K_n))$ then $I$ is not a maximal independent set. Thus $|A|=\alpha (T_2(K_n))$.
 
Suppose $|B|<m$. Then there is at least one vertex $x'\in X$ that is in no pair of $B$. Let $M\subseteq Y$ be the set of vertices of $Y$ that are not
in any pair of $B$. Then $M$ contains at least
$n-m+1\geq t+2$ vertices.

Suppose $x'=x_1$. Then at most $t$ vertices 
$x'y\in XY$ are adjacent to vertices in $A$. Additionally, at most one vertex  
$x'y\in XY$ is adjacent to a vertex in $C$. Thus there
is some $y'\in M$ such that $x'y'$, is not adjacent to any vertex of $A$ or $C$. Thus $\{x'y'\}\cup I$ is an independent set.

Suppose $x'\neq x_1$. If $x'x_1\in C$, then at most $t$ vertices  $x_1y\in XY$ are adjacent to $x'x_1$. Additionally, there is at most one vertex  $yy_1\in A$. Thus there is at least one vertex, say $x'y'$, which is not adjacent to any vertex in $A$ or $C$. Then $\{x'y'\}\cup I$ is an independent set. 

Suppose $x'x_1\not\in C$. Then there is at most one vertex $x''\in X$ such that $x'x'' \in C$. Additionally, there is at most one vertex $yy_1 \in A$. Thus there exists at least one vertex of the form $x'y'$ such that $\{x'y'\}\cup I$ is an independent set. 

In each case, we have seen that if $|B|<m$, then
$I$ is not maximal. Thus $|B|=m$.

Suppose $|C|<\alpha (T_2(K_m))$. A similar argument to that used for $A$ shows that $|C|=\alpha (T_2(K_m))$.
Thus $T_2(G)$ is well-covered.
\end{proof}

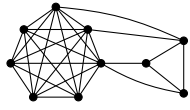
\begin{figure}
\centering
\begin{tikzpicture}
\node (1) at (-0.6,0.15)[place][]{};
\node (2) at (-0.425,0.6)[place][]{};
\node (3) at (0,0.9)[place][]{};
\node (4) at (0.6,0.15)[place][]{};
\node (5) at (0.425,0.6)[place][]{};
\node (6) at (0.3,-0.3)[place][]{};
\node (7) at (-0.3,-0.3)[place][]{};
\node (8) at (1.2,.15)[place][]{};

\node (9) at (1.7,0.45)[place]{};
\node (10) at (1.7,-0.25)[place]{};

\draw (8) to (9);
\draw (10) to (9);
\draw (8) to (10);

\draw [] (1) to (2);
\draw [] (1) to (3);
\draw [] (1) to (4);
\draw [] (1) to (5);
\draw [] (1) to (6);
\draw [] (1) to (7);
\draw [] (2) to (3);
\draw [] (2) to (4);
\draw [] (2) to (5);
\draw [] (2) to (6);
\draw [] (2) to (7);
\draw [] (3) to (4);
\draw [] (3) to (5);
\draw [] (3) to (6);
\draw [] (3) to (7);
\draw [] (4) to (5);
\draw [] (4) to (6);
\draw [] (4) to (7);
\draw [] (5) to (6);
\draw [] (5) to (7);
\draw [] (6) to (7);

\draw [] (4) to (8);
\draw [] (4) to [bend right=10] (10);

\draw [] (9) to [bend right=10] (3);
\draw [] (9) to (5);
\end{tikzpicture}

\caption{A graph $G$ with $T_2(G)$ well-covered by Theorem~\ref{thm:bbc}.}\label{fig:bb53e}
\end{figure}

\section{Concluding comments}

By computer calculation, one can 
check that if
 $G$ is a graph on eight or fewer vertices and
 $T_2(G)$  
well-covered, then
 $G\in\G$. In fact, 
 all these graphs are accounted for by  Theorems~\ref{thm:complete}, \ref{thm:bba}
 and \ref{thm:bbb} (see the Appendix). The graphs covered in Theorem~\ref{thm:bbc} must have
 at least ten vertices, such as in 
 Figure~\ref{fig:bb53e}.
 
The theorems in
Section~\ref{sec:wcgraphs} considered graphs in $\G$ when $n>m$. We do not expect that $T_2(G)$ is well-covered
for any non-complete graph  $G\in \G$ with $n=m$. The following theorem is an  illustration. 

\begin{theorem} Suppose $G\in \G$ and $m=n>1$ and $x_1y_1$ is the only edge with one endpoint
in $Y$ and one in $X$. Then $T_2(G)$ is not well-covered.
\end{theorem}
\begin{proof}
 Let $A=\{ x_1x_2, x_3x_4,\ldots\} $ and $B= \{ y_1y_2,y_3y_4,\ldots \}$
and $C =\{ x_1y_1,x_2y_2,\ldots x_ny_n\}$ then $A\cup B\cup C$ is 
an independent set and so by Remark~\ref{rem:parts}, 
$\alpha(T_2(G))=2\alpha(T_2(K_n)) +\alpha(K_n \square K_n)$. 
 Let $I= \{ x_1x_2, x_3x_4,\ldots\} \cup \{ y_1y_2,y_3y_4,\ldots \}  \cup \{x_1y_{n}\} \cup \{x_3y_2,\ldots, x_ny_{n-1} \}$, 
then $\vert I\vert =\alpha(T_2(G))-1$ and yet $I$ is maximal. Therefore $T_2(G)$ is not well-covered.
\end{proof}

\begin{remark}
If $G \in \G$ and $m=n$ and there is at least one edge $xy$ with one endpoint in $Y$ and one in $X$
and $\alpha(T_2(G))=\alpha(T_2(G)[XX])+\alpha(T_2(G)[XY])+\alpha(T_2(G)[YY])=2\alpha(T_2(K_n))+n$, then $T_2(G)$ is not well-covered. In particular, suppose $x_1y_1$ is an 
edge of $G$. Let $I$ be a maximal
independent set of $T_2(G)$ with $\{x_1y_1\} \cup\{x_1y_{n}\} \cup \{x_3y_2,\ldots, x_ny_{n-1} \} \subseteq I$. Then
$|I\cap XY|=n-1 < \alpha(T_2(G)[XY])$ and so $|I| \neq \alpha(T_2(G))$. 
\end{remark}

 One of the reasons that we are interested in well-covered graphs is that they are 
 candidate Cohen-Macaulay graphs (for details and definitions, see e.g. \cite{EVV}). 
 As an example, we can show that if $G$ is a non-complete graph $G$ of
 order $4$ with $T_2(G)$ well-covered, then $T_2(G)$ is vertex-deomposable 
 and hence $T_2(G)$ is Cohen-Macaulay. Future work could be done
 to determine when a well-covered token graph is vertex-decomposable and/or
 Cohen-Macaulay.

\newpage 

\section{Appendix: Graphs $G$ with $T_2(G)$ well-covered.}

The number of graphs $G$ of order at most $9$ with $T_2(G)$ well-covered are listed
in Table~\ref{tab:count} as determined by a computer search. The following figures display all the non-complete graphs $G$ of order at most
$9$ with $T_2(G)$ well-covered.
\begin{table}[ht]
\centering
\begin{tabular}{|c||c|c|c|c|c|c|c|c|}
\hline 
$n$ &2&3&4&5&6&7&8&9\\ \hline 
number of graphs &1&1&3&1&5&1&13&9\\ \hline 
\end{tabular}\vspace{0.1in}

\caption{Number of graphs $G$ of order $n$ with $T_2(G)$ well-covered for $n\leq 9$.}\label{tab:count}
\end{table}\vspace{-0.4in}
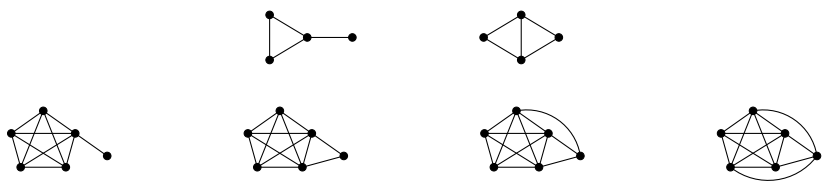
\begin{figure}[ht] 
\centering 
\begin{tikzpicture}
\node (1) at (-0.5,-0.3)[place][]{};
\node (2) at (-0.5,0.3)[place][]{};
\node (3) at (0,0)[place][]{};
\node (4) at (0.6,0)[place][]{};
\draw[] (1) to (2);
\draw [] (1) to (3);
\draw [] (2) to (3);
\draw [] (3) to (4);
\end{tikzpicture}
\hspace{15mm}\vspace{0.2in}
\begin{tikzpicture}
\node (1) at (-0.5,0)[place][]{};
\node (2) at (0,0.3)[place][]{};
\node (3) at (0,-0.3)[place][]{};
\node (4) at (0.5,0)[place][]{};
\draw [] (1) to (2);
\draw [] (1) to (3);
\draw [] (2) to (3);
\draw [] (3) to (4);
\draw [] (2) to (4);
\end{tikzpicture}\\
\begin{tikzpicture}
\node (1) at (-0.425,0.15)[place][]{};
\node (2) at (0,0.45)[place][]{};
\node (3) at (0.425,0.15)[place][]{};
\node (4) at (0.3,-0.3)[place][]{};
\node (5) at (-0.3,-0.3)[place][]{};
\node (6) at (0.85,-0.15)[place][]{};
\draw [] (1) to (2);
\draw [] (1) to (3);
\draw [] (1) to (4);
\draw [] (1) to (5);
\draw [] (2) to (3);
\draw [] (2) to (4);
\draw [] (2) to (5);
\draw [] (3) to (4);
\draw [] (3) to (5);
\draw [] (4) to (5);
\draw [] (3) to (6);
\phantom{\draw [] (5) to [bend right=45] (6);}
\end{tikzpicture}
\hspace{15mm}
\begin{tikzpicture}
\node (1) at (-0.425,0.15)[place][]{};
\node (2) at (0,0.45)[place][]{};
\node (3) at (0.425,0.15)[place][]{};
\node (4) at (0.3,-0.3)[place][]{};
\node (5) at (-0.3,-0.3)[place][]{};
\node (6) at (0.85,-0.15)[place][]{};
\draw [] (1) to (2);
\draw [] (1) to (3);
\draw [] (1) to (4);
\draw [] (1) to (5);
\draw [] (2) to (3);
\draw [] (2) to (4);
\draw [] (2) to (5);
\draw [] (3) to (4);
\draw [] (3) to (5);
\draw [] (4) to (5);
\draw [] (3) to (6);
\draw [] (4) to (6);
\phantom{\draw [] (5) to [bend right=45] (6);}
\end{tikzpicture}
\hspace{15mm}
\begin{tikzpicture}
\node (1) at (-0.425,0.15)[place][]{};
\node (2) at (0,0.45)[place][]{};
\node (3) at (0.425,0.15)[place][]{};
\node (4) at (0.3,-0.3)[place][]{};
\node (5) at (-0.3,-0.3)[place][]{};
\node (6) at (0.85,-0.15)[place][]{};
\draw [] (1) to (2);
\draw [] (1) to (3);
\draw [] (1) to (4);
\draw [] (1) to (5);
\draw [] (2) to (3);
\draw [] (2) to (4);
\draw [] (2) to (5);
\draw [] (3) to (4);
\draw [] (3) to (5);
\draw [] (4) to (5);
\draw [] (3) to (6);
\draw [] (4) to (6);
\draw [] (2) to [bend left=45] (6);
\phantom{\draw [] (5) to [bend right=45] (6);}
\end{tikzpicture}
\hspace{15mm}
\begin{tikzpicture}
\node (1) at (-0.425,0.15)[place][]{};
\node (2) at (0,0.45)[place][]{};
\node (3) at (0.425,0.15)[place][]{};
\node (4) at (0.3,-0.3)[place][]{};
\node (5) at (-0.3,-0.3)[place][]{};
\node (6) at (0.85,-0.15)[place][]{};
\draw [] (1) to (2);
\draw [] (1) to (3);
\draw [] (1) to (4);
\draw [] (1) to (5);
\draw [] (2) to (3);
\draw [] (2) to (4);
\draw [] (2) to (5);
\draw [] (3) to (4);
\draw [] (3) to (5);
\draw [] (4) to (5);
\draw [] (3) to (6);
\draw [] (4) to (6);\vspace{0.2in}
\draw [] (2) to [bend left=45] (6);
\draw [] (5) to [bend right=45] (6);
\end{tikzpicture}
\caption{Non-complete graphs $G$ of order $4$ and $6$ with $T_2(G)$ well-covered.}
\end{figure}
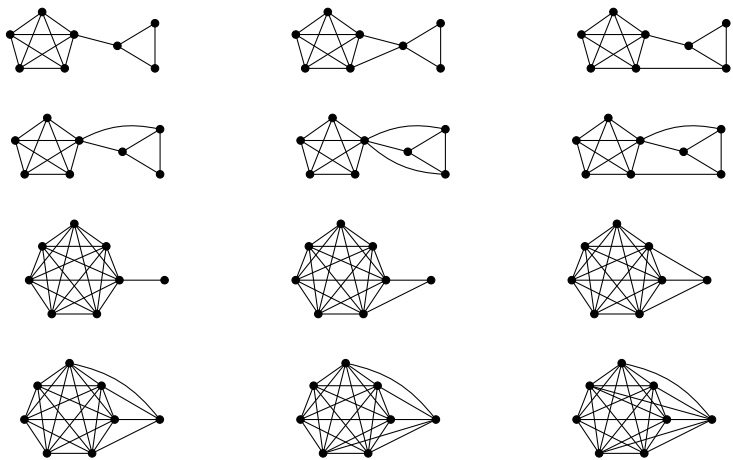
\begin{figure}[ht!]
\centering
\begin{tikzpicture}
\node (1) at (-0.425,0.15)[place][]{};
\node (2) at (0,0.45)[place][]{};
\node (3) at (0.425,0.15)[place][]{};
\node (4) at (0.3,-0.3)[place][]{};
\node (5) at (-0.3,-0.3)[place][]{};
\node (6) at (1,0)[place][]{};
\node (7) at (1.5,0.3)[place][]{};
\node (8) at (1.5,-0.3)[place][]{};
\draw [] (1) to (2);
\draw [] (1) to (3);
\draw [] (1) to (4);
\draw [] (1) to (5);
\draw [] (2) to (3);
\draw [] (2) to (4);
\draw [] (2) to (5);
\draw [] (3) to (4);
\draw [] (3) to (5);
\draw [] (4) to (5);
\draw [] (3) to (6);
\draw [] (6) to (7);
\draw [] (6) to (8);
\draw [] (7) to (8);
\end{tikzpicture}
\hspace{15mm} \vspace{0.2in}
\begin{tikzpicture}
\node (1) at (-0.425,0.15)[place][]{};
\node (2) at (0,0.45)[place][]{};
\node (3) at (0.425,0.15)[place][]{};
\node (4) at (0.3,-0.3)[place][]{};
\node (5) at (-0.3,-0.3)[place][]{};
\node (6) at (1,0)[place][]{};
\node (7) at (1.5,0.3)[place][]{};
\node (8) at (1.5,-0.3)[place][]{};
\draw [] (1) to (2);
\draw [] (1) to (3);
\draw [] (1) to (4);
\draw [] (1) to (5);
\draw [] (2) to (3);
\draw [] (2) to (4);
\draw [] (2) to (5);
\draw [] (3) to (4);
\draw [] (3) to (5);
\draw [] (4) to (5);
\draw [] (3) to (6);
\draw [] (6) to (7);
\draw [] (6) to (8);
\draw [] (7) to (8);
\draw [] (4) to (6);
\end{tikzpicture}
\hspace{15mm}
\begin{tikzpicture}
\node (1) at (-0.425,0.15)[place][]{};
\node (2) at (0,0.45)[place][]{};
\node (3) at (0.425,0.15)[place][]{};
\node (4) at (0.3,-0.3)[place][]{};
\node (5) at (-0.3,-0.3)[place][]{};
\node (6) at (1,0)[place][]{};
\node (7) at (1.5,0.3)[place][]{};
\node (8) at (1.5,-0.3)[place][]{};
\draw [] (1) to (2);
\draw [] (1) to (3);
\draw [] (1) to (4);
\draw [] (1) to (5);
\draw [] (2) to (3);
\draw [] (2) to (4);
\draw [] (2) to (5);
\draw [] (3) to (4);
\draw [] (3) to (5);
\draw [] (4) to (5);
\draw [] (3) to (6);
\draw [] (6) to (7);
\draw [] (6) to (8);
\draw [] (7) to (8);
\draw [] (4) to (8);
\end{tikzpicture}\\  
\centering
\begin{tikzpicture}
\node (1) at (-0.425,0.15)[place][]{};
\node (2) at (0,0.45)[place][]{};
\node (3) at (0.425,0.15)[place][]{};
\node (4) at (0.3,-0.3)[place][]{};
\node (5) at (-0.3,-0.3)[place][]{};
\node (6) at (1,0)[place][]{};
\node (7) at (1.5,0.3)[place][]{};
\node (8) at (1.5,-0.3)[place][]{};
\draw [] (1) to (2);
\draw [] (1) to (3);
\draw [] (1) to (4);
\draw [] (1) to (5);
\draw [] (2) to (3);
\draw [] (2) to (4);
\draw [] (2) to (5);
\draw [] (3) to (4);
\draw [] (3) to (5);
\draw [] (4) to (5);
\draw [] (3) to (6);
\draw [] (3) to [bend left=20] (7);
\draw [] (6) to (7);
\draw [] (6) to (8);
\draw [] (7) to (8);
\end{tikzpicture}
\hspace{15mm}
\begin{tikzpicture}
\node (1) at (-0.425,0.15)[place][]{};
\node (2) at (0,0.45)[place][]{};
\node (3) at (0.425,0.15)[place][]{};
\node (4) at (0.3,-0.3)[place][]{};
\node (5) at (-0.3,-0.3)[place][]{};
\node (6) at (1,0)[place][]{};
\node (7) at (1.5,0.3)[place][]{};
\node (8) at (1.5,-0.3)[place][]{};
\draw [] (1) to (2);
\draw [] (1) to (3);
\draw [] (1) to (4);
\draw [] (1) to (5);
\draw [] (2) to (3);
\draw [] (2) to (4);
\draw [] (2) to (5);
\draw [] (3) to (4);
\draw [] (3) to (5);
\draw [] (4) to (5);
\draw [] (3) to (6);
\draw [] (6) to (7);
\draw [] (6) to (8);
\draw [] (7) to (8);
\draw [] (3) to [bend left=20] (7);
\draw [] (3) to [bend right=20] (8);
\end{tikzpicture}
\hspace{15mm}\vspace{0.2in}
\begin{tikzpicture}
\node (1) at (-0.425,0.15)[place][]{};
\node (2) at (0,0.45)[place][]{};
\node (3) at (0.425,0.15)[place][]{};
\node (4) at (0.3,-0.3)[place][]{};
\node (5) at (-0.3,-0.3)[place][]{};
\node (6) at (1,0)[place][]{};
\node (7) at (1.5,0.3)[place][]{};
\node (8) at (1.5,-0.3)[place][]{};
\draw [] (1) to (2);
\draw [] (1) to (3);
\draw [] (1) to (4);
\draw [] (1) to (5);
\draw [] (2) to (3);
\draw [] (2) to (4);
\draw [] (2) to (5);
\draw [] (3) to (4);
\draw [] (3) to (5);
\draw [] (4) to (5);
\draw [] (3) to (6);
\draw [] (6) to (7);
\draw [] (6) to (8);
\draw [] (7) to (8);
\draw [] (3) to [bend left=20] (7);
\draw [] (4) to (8);
\end{tikzpicture}\\
\begin{tikzpicture}
\node (1) at (-0.6,0.15)[place][]{};
\node (2) at (-0.425,0.6)[place][]{};
\node (3) at (0,0.9)[place][]{};
\node (4) at (0.6,0.15)[place][]{};
\node (5) at (0.425,0.6)[place][]{};
\node (6) at (0.3,-0.3)[place][]{};
\node (7) at (-0.3,-0.3)[place][]{};
\node (8) at (1.2,.15)[place][]{};

\draw [] (1) to (2);
\draw [] (1) to (3);
\draw [] (1) to (4);
\draw [] (1) to (5);
\draw [] (1) to (6);
\draw [] (1) to (7);
\draw [] (2) to (3);
\draw [] (2) to (4);
\draw [] (2) to (5);
\draw [] (2) to (6);
\draw [] (2) to (7);
\draw [] (3) to (4);
\draw [] (3) to (5);
\draw [] (3) to (6);
\draw [] (3) to (7);
\draw [] (4) to (5);
\draw [] (4) to (6);
\draw [] (4) to (7);
\draw [] (5) to (6);
\draw [] (5) to (7);
\draw [] (6) to (7);

\draw [] (4) to (8);
\end{tikzpicture}
\hspace{15mm}\vspace{0.2in}
\begin{tikzpicture}
\node (1) at (-0.6,0.15)[place][]{};
\node (2) at (-0.425,0.6)[place][]{};
\node (3) at (0,0.9)[place][]{};
\node (4) at (0.6,0.15)[place][]{};
\node (5) at (0.425,0.6)[place][]{};
\node (6) at (0.3,-0.3)[place][]{};
\node (7) at (-0.3,-0.3)[place][]{};
\node (8) at (1.2,.15)[place][]{};
\draw [] (1) to (2);
\draw [] (1) to (3);
\draw [] (1) to (4);
\draw [] (1) to (5);
\draw [] (1) to (6);
\draw [] (1) to (7);
\draw [] (2) to (3);
\draw [] (2) to (4);
\draw [] (2) to (5);
\draw [] (2) to (6);
\draw [] (2) to (7);
\draw [] (3) to (4);
\draw [] (3) to (5);
\draw [] (3) to (6);
\draw [] (3) to (7);
\draw [] (4) to (5);
\draw [] (4) to (6);
\draw [] (4) to (7);
\draw [] (5) to (6);
\draw [] (5) to (7);
\draw [] (6) to (7);

\draw [] (4) to (8);
\draw [] (6) to (8);

\end{tikzpicture}
\hspace{15mm}
\begin{tikzpicture}
\node (1) at (-0.6,0.15)[place][]{};
\node (2) at (-0.425,0.6)[place][]{};
\node (3) at (0,0.9)[place][]{};
\node (4) at (0.6,0.15)[place][]{};
\node (5) at (0.425,0.6)[place][]{};
\node (6) at (0.3,-0.3)[place][]{};
\node (7) at (-0.3,-0.3)[place][]{};
\node (8) at (1.2,.15)[place][]{};
\draw [] (1) to (2);
\draw [] (1) to (3);
\draw [] (1) to (4);
\draw [] (1) to (5);
\draw [] (1) to (6);
\draw [] (1) to (7);
\draw [] (2) to (3);
\draw [] (2) to (4);
\draw [] (2) to (5);
\draw [] (2) to (6);
\draw [] (2) to (7);
\draw [] (3) to (4);
\draw [] (3) to (5);
\draw [] (3) to (6);
\draw [] (3) to (7);
\draw [] (4) to (5);
\draw [] (4) to (6);
\draw [] (4) to (7);
\draw [] (5) to (6);
\draw [] (5) to (7);
\draw [] (6) to (7);

\draw [] (4) to (8);
\draw [] (5) to (8);
\draw [] (6) to (8);

\end{tikzpicture} \\
\begin{tikzpicture}
\node (1) at (-0.6,0.15)[place][]{};
\node (2) at (-0.425,0.6)[place][]{};
\node (3) at (0,0.9)[place][]{};
\node (4) at (0.6,0.15)[place][]{};
\node (5) at (0.425,0.6)[place][]{};
\node (6) at (0.3,-0.3)[place][]{};
\node (7) at (-0.3,-0.3)[place][]{};
\node (8) at (1.2,.15)[place][]{};
\draw [] (1) to (2);
\draw [] (1) to (3);
\draw [] (1) to (4);
\draw [] (1) to (5);
\draw [] (1) to (6);
\draw [] (1) to (7);
\draw [] (2) to (3);
\draw [] (2) to (4);
\draw [] (2) to (5);
\draw [] (2) to (6);
\draw [] (2) to (7);
\draw [] (3) to (4);
\draw [] (3) to (5);
\draw [] (3) to (6);
\draw [] (3) to (7);
\draw [] (4) to (5);
\draw [] (4) to (6);
\draw [] (4) to (7);
\draw [] (5) to (6);
\draw [] (5) to (7);
\draw [] (6) to (7);

\draw [] (4) to (8);
\draw [] (5) to (8);
\draw [] (6) to (8);
\draw [] (3) to [bend left=20] (8);
\end{tikzpicture} 
\hspace{15mm}
\begin{tikzpicture}
\node (1) at (-0.6,0.15)[place][]{};
\node (2) at (-0.425,0.6)[place][]{};
\node (3) at (0,0.9)[place][]{};
\node (4) at (0.6,0.15)[place][]{};
\node (5) at (0.425,0.6)[place][]{};
\node (6) at (0.3,-0.3)[place][]{};
\node (7) at (-0.3,-0.3)[place][]{};
\node (8) at (1.2,.15)[place][]{};
\draw [] (1) to (2);
\draw [] (1) to (3);
\draw [] (1) to (4);
\draw [] (1) to (5);
\draw [] (1) to (6);
\draw [] (1) to (7);
\draw [] (2) to (3);
\draw [] (2) to (4);
\draw [] (2) to (5);
\draw [] (2) to (6);
\draw [] (2) to (7);
\draw [] (3) to (4);
\draw [] (3) to (5);
\draw [] (3) to (6);
\draw [] (3) to (7);
\draw [] (4) to (5);
\draw [] (4) to (6);
\draw [] (4) to (7);
\draw [] (5) to (6);
\draw [] (5) to (7);
\draw [] (6) to (7);

\draw [] (4) to (8);
\draw [] (5) to (8);
\draw [] (6) to (8);
\draw [] (3) to [bend left=20] (8);
\draw [] (7) to (8);
\end{tikzpicture} 
\hspace{15mm}
\begin{tikzpicture}
\node (1) at (-0.6,0.15)[place][]{};
\node (2) at (-0.425,0.6)[place][]{};
\node (3) at (0,0.9)[place][]{};
\node (4) at (0.6,0.15)[place][]{};
\node (5) at (0.425,0.6)[place][]{};
\node (6) at (0.3,-0.3)[place][]{};
\node (7) at (-0.3,-0.3)[place][]{};
\node (8) at (1.2,.15)[place][]{};
\draw [] (1) to (2);
\draw [] (1) to (3);
\draw [] (1) to (4);
\draw [] (1) to (5);
\draw [] (1) to (6);
\draw [] (1) to (7);
\draw [] (2) to (3);
\draw [] (2) to (4);
\draw [] (2) to (5);
\draw [] (2) to (6);
\draw [] (2) to (7);
\draw [] (3) to (4);
\draw [] (3) to (5);
\draw [] (3) to (6);
\draw [] (3) to (7);
\draw [] (4) to (5);
\draw [] (4) to (6);
\draw [] (4) to (7);
\draw [] (5) to (6);
\draw [] (5) to (7);
\draw [] (6) to (7);

\draw [] (4) to (8);
\draw [] (5) to (8);
\draw [] (6) to (8);
\draw [] (3) to [bend left=20] (8);
\draw [] (7) to (8);
\draw [] (2) to (8);
\end{tikzpicture} 
\caption{Non-complete graphs $G$ of order $8$ with $T_2(G)$ well-covered.}
\centering
\end{figure}


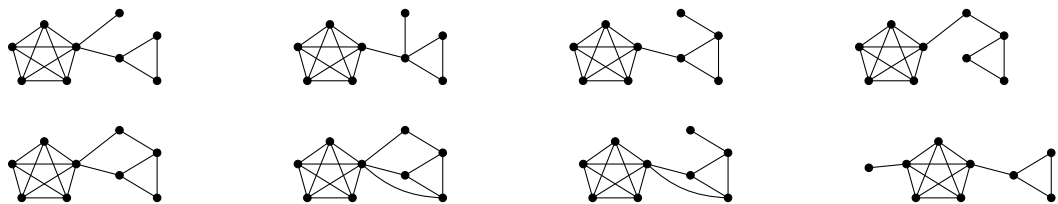
\begin{figure}[ht!]
\centering
\begin{tikzpicture}
\node (1) at (-0.425,0.15)[place]{};
\node (2) at (0,0.45)[place]{};
\node (3) at (0.425,0.15)[place]{};
\node (4) at (0.3,-0.3)[place]{};
\node (5) at (-0.3,-0.3)[place]{};
\node (6) at (1,0)[place]{};
\node (7) at (1.5,0.3)[place]{};
\node (8) at (1.5,-0.3)[place]{};
\node (9) at (1,0.6)[place]{};
\draw [] (1) to (2);
\draw [] (1) to (3);
\draw [] (1) to (4);
\draw [] (1) to (5);
\draw [] (2) to (3);
\draw [] (2) to (4);
\draw [] (2) to (5);
\draw [] (3) to (4);
\draw [] (3) to (5);
\draw [] (4) to (5);
\draw [] (3) to (6);
\draw [] (6) to (7);
\draw [] (6) to (8);
\draw [] (7) to (8);

\draw [] (3) to (9);
\end{tikzpicture}
\hspace{15mm}
\begin{tikzpicture}
\node (1) at (-0.425,0.15)[place]{};
\node (2) at (0,0.45)[place]{};
\node (3) at (0.425,0.15)[place]{};
\node (4) at (0.3,-0.3)[place]{};
\node (5) at (-0.3,-0.3)[place]{};
\node (6) at (1,0)[place]{};
\node (7) at (1.5,0.3)[place]{};
\node (8) at (1.5,-0.3)[place]{};
\node (9) at (1,0.6)[place]{};
\draw [] (1) to (2);
\draw [] (1) to (3);
\draw [] (1) to (4);
\draw [] (1) to (5);
\draw [] (2) to (3);
\draw [] (2) to (4);
\draw [] (2) to (5);
\draw [] (3) to (4);
\draw [] (3) to (5);
\draw [] (4) to (5);
\draw [] (3) to (6);
\draw [] (6) to (7);
\draw [] (6) to (8);
\draw [] (7) to (8);

\draw [] (6) to (9);
\end{tikzpicture}
\hspace{15mm}\vspace{.2in}
\begin{tikzpicture}
\node (1) at (-0.425,0.15)[place]{};
\node (2) at (0,0.45)[place]{};
\node (3) at (0.425,0.15)[place]{};
\node (4) at (0.3,-0.3)[place]{};
\node (5) at (-0.3,-0.3)[place]{};
\node (6) at (1,0)[place]{};
\node (7) at (1.5,0.3)[place]{};
\node (8) at (1.5,-0.3)[place]{};
\node (9) at (1,0.6)[place]{};
\draw [] (1) to (2);
\draw [] (1) to (3);
\draw [] (1) to (4);
\draw [] (1) to (5);
\draw [] (2) to (3);
\draw [] (2) to (4);
\draw [] (2) to (5);
\draw [] (3) to (4);
\draw [] (3) to (5);
\draw [] (4) to (5);
\draw [] (3) to (6);
\draw [] (6) to (7);
\draw [] (6) to (8);
\draw [] (7) to (8);

\draw [] (7) to (9);
\end{tikzpicture}
\hspace{15mm}
\begin{tikzpicture}
\node (1) at (-0.425,0.15)[place]{};
\node (2) at (0,0.45)[place]{};
\node (3) at (0.425,0.15)[place]{};
\node (4) at (0.3,-0.3)[place]{};
\node (5) at (-0.3,-0.3)[place]{};
\node (6) at (1,0)[place]{};
\node (7) at (1.5,0.3)[place]{};
\node (8) at (1.5,-0.3)[place]{};
\node (9) at (1,0.6)[place]{};
\draw [] (1) to (2);
\draw [] (1) to (3);
\draw [] (1) to (4);
\draw [] (1) to (5);
\draw [] (2) to (3);
\draw [] (2) to (4);
\draw [] (2) to (5);
\draw [] (3) to (4);
\draw [] (3) to (5);
\draw [] (4) to (5);
\draw [] (6) to (7);
\draw [] (6) to (8);
\draw [] (7) to (8);

\draw [] (3) to (9);
\draw [] (7) to (9);
\end{tikzpicture}\\
\hspace{5mm}
\begin{tikzpicture}
\node (1) at (-0.425,0.15)[place]{};
\node (2) at (0,0.45)[place]{};
\node (3) at (0.425,0.15)[place]{};
\node (4) at (0.3,-0.3)[place]{};
\node (5) at (-0.3,-0.3)[place]{};
\node (6) at (1,0)[place]{};
\node (7) at (1.5,0.3)[place]{};
\node (8) at (1.5,-0.3)[place]{};
\node (9) at (1,0.6)[place]{};
\draw [] (1) to (2);
\draw [] (1) to (3);
\draw [] (1) to (4);
\draw [] (1) to (5);
\draw [] (2) to (3);
\draw [] (2) to (4);
\draw [] (2) to (5);
\draw [] (3) to (4);
\draw [] (3) to (5);
\draw [] (4) to (5);
\draw [] (3) to (6);
\draw [] (6) to (7);
\draw [] (6) to (8);
\draw [] (7) to (8);

\draw [] (3) to (9);
\draw [] (7) to (9);
\end{tikzpicture}
\hspace{15mm}
\begin{tikzpicture}
\node (1) at (-0.425,0.15)[place]{};
\node (2) at (0,0.45)[place]{};
\node (3) at (0.425,0.15)[place]{};
\node (4) at (0.3,-0.3)[place]{};
\node (5) at (-0.3,-0.3)[place]{};
\node (6) at (1,0)[place]{};
\node (7) at (1.5,0.3)[place]{};
\node (8) at (1.5,-0.3)[place]{};
\node (9) at (1,0.6)[place]{};
\draw [] (1) to (2);
\draw [] (1) to (3);
\draw [] (1) to (4);
\draw [] (1) to (5);
\draw [] (2) to (3);
\draw [] (2) to (4);
\draw [] (2) to (5);
\draw [] (3) to (4);
\draw [] (3) to (5);
\draw [] (4) to (5);
\draw [] (3) to (6);
\draw [] (6) to (7);
\draw [] (6) to (8);
\draw [] (7) to (8);

\draw [] (3) to (9);
\draw [] (7) to (9);
\draw [] (3) to [bend right=20] (8);

\end{tikzpicture}
\hspace{15mm}
\begin{tikzpicture}
\node (1) at (-0.425,0.15)[place]{};
\node (2) at (0,0.45)[place]{};
\node (3) at (0.425,0.15)[place]{};
\node (4) at (0.3,-0.3)[place]{};
\node (5) at (-0.3,-0.3)[place]{};
\node (6) at (1,0)[place]{};
\node (7) at (1.5,0.3)[place]{};
\node (8) at (1.5,-0.3)[place]{};
\node (9) at (1,0.6)[place]{};
\draw [] (1) to (2);
\draw [] (1) to (3);
\draw [] (1) to (4);
\draw [] (1) to (5);
\draw [] (2) to (3);
\draw [] (2) to (4);
\draw [] (2) to (5);
\draw [] (3) to (4);
\draw [] (3) to (5);
\draw [] (4) to (5);
\draw [] (3) to (6);
\draw [] (6) to (7);
\draw [] (6) to (8);
\draw [] (7) to (8);

\draw [] (7) to (9);
\draw [] (3) to [bend right=20] (8);
\end{tikzpicture}
\hspace{15mm}
\begin{tikzpicture}
\node (9) at (-0.925, 0.1)[place]{};
\node (1) at (-0.425,0.15)[place]{};
\node (2) at (0,0.45)[place]{};
\node (3) at (0.425,0.15)[place]{};
\node (4) at (0.3,-0.3)[place]{};
\node (5) at (-0.3,-0.3)[place]{};
\node (6) at (1,0)[place]{};
\node (7) at (1.5,0.3)[place]{};
\node (8) at (1.5,-0.3)[place]{};
\draw [] (1) to (2);
\draw [] (1) to (3);
\draw [] (1) to (4);
\draw [] (1) to (5);
\draw [] (2) to (3);
\draw [] (2) to (4);
\draw [] (2) to (5);
\draw [] (3) to (4);
\draw [] (3) to (5);
\draw [] (4) to (5);
\draw [] (6) to (7);
\draw [] (6) to (8);
\draw [] (7) to (8);

\draw [] (1) to (9);
\draw [] (3) to (6);
\end{tikzpicture}
\caption{Non-complete graphs $G$ of order $9$ with $T_2(G)$ well-covered.\bigskip\bigskip}
\end{figure}
\bigskip
\bigskip
\hrulefill{}

\end{document}